\documentclass[12pt, a4paper]{amsart}
\usepackage[dvips]{epsfig}
\usepackage{amsmath}
\usepackage{amssymb}
\usepackage{amsfonts}
\usepackage{bbm}
\usepackage{amsthm}
\usepackage{amsbsy}
\usepackage{amsgen}
\usepackage{amscd}
\usepackage{amsopn}
\usepackage{amstext}
\usepackage{amsxtra}
\usepackage[utf8]{inputenc}
\usepackage{enumerate}
\usepackage{hyperref}
\usepackage{lineno}
\usepackage{marginnote}
\usepackage{verbatim}
\usepackage{calrsfs}
\usepackage{times, graphicx}
\usepackage{layout}
\usepackage{subcaption}
\usepackage{eso-pic}
\usepackage{lastpage}
\usepackage{dsfont}
\usepackage{fixmath}
\usepackage{extarrows}

\newcommand{\tmop}[1]{\ensuremath{\operatorname{#1}}}

\DeclareMathAlphabet{\pazocal}{OMS}{zplm}{m}{n}
\usepackage[foot]{amsaddr}

  \evensidemargin
\oddsidemargin
\newtheorem{theorem}{Theorem}[section]
\newtheorem{proposition}[theorem]{Proposition}
\newtheorem{lemma}[theorem]{Lemma}
\newtheorem{corollary}[theorem]{Corollary}

\theoremstyle{definition}

\newcommand {\Z} {\mathbb{Z}}

\newcommand {\E} {\mathbb{E}}

\newcommand {\C} {\mathbb {C}}

\newcommand {\M} {\pazocal {M}}

\newcommand{\prob}{\pazocal{P}r}

\newcommand {\Lc} {\pazocal {L}}

\newcommand {\Hb} {\mathbb {H}}
\newcommand {\Dc} {\mathcal{D}}
\newcommand {\Kc} {\mathcal{K}}

\newcommand {\Bc} {\pazocal {B}}

\newcommand {\Fc} {\pazocal{F}}

\newcommand{\var}{\operatorname{Var}}

\newcommand{\Sc}{{\pazocal{S}}}
\newcommand{\Scc}{{\mathcal{S}}}

\newcommand{\lk}{{L}}

\numberwithin{equation}{section}

\title[Supremum of random cusp forms]{On the supremum of random cusp forms }
\author[Huang]{Bingrong Huang}
\address{
   Data Science Institute and School of Mathematics \\
   Shandong University \\
   Jinan \\
   Shandong 250100 \\
China
}
\email{brhuang@sdu.edu.cn}

\author[Lester]{Stephen Lester}
\address{
   Department of Mathematics \\
   King's College London \\
   London WC2R 2LS \\
   UK
}
\email{steve.lester@kcl.ac.uk}

\author[Wigman]{Igor Wigman}
\address{
   Department of Mathematics \\
   King's College London \\
   London WC2R 2LS \\
   UK
}
\email{igor.wigman@kcl.ac.uk}

\author[Yesha]{Nadav Yesha}
\address{
   Department of Mathematics \\
   University of Haifa \\
   3103301 Haifa \\
   Israel
}
\email{nyesha@univ.haifa.ac.il}

\begin{document}
\begin{abstract}

A random ensemble of cusp forms for the full modular group is introduced. For a weight-$k$ cusp form, restricted to a compact
subdomain of the modular surface, the true order of magnitude of its expected supremum is determined to be $\asymp \sqrt{\log{k}}$, in line with the conjectured bounds.
Additionally, the exponential concentration of the supremum around its median is established. Contrary to the compact case, it is shown that the global expected supremum, which is attained around the cusp,
grows like $k^{1/4}$, up to a logarithmic factor.

\end{abstract}
\maketitle

\section{Introduction}

\subsection{Random cusp forms}

For $k\ge 2$ even let $\Sc_{k}$ be the space of weight-$k$ cusp forms for the full modular group $\Gamma=\tmop{SL}_2(\mathbb Z)$ of dimension $N=N_{k}=\dim\Sc_{k}$. It is well-known that $N_{k}=\frac{k}{12}+O(1)$ and that $N_k$ is explicitly given in terms of
$\lfloor \frac{k}{12}\rfloor$, depending on $k\mod 12$.
Naturally, a weight-$k$ cusp form $f\in \Sc_{k}$ is an analytic function $f:\Dc\rightarrow\C$,
where $$\Dc = \left\{ z\in\C:\: |\Re(z)|\le \frac{1}{2}, |z|\ge 1 \right\}\subseteq \Hb^{2}$$ is the canonical (closed) fundamental domain for the action of $\Gamma$ on the Poincar\'{e} half-plane model
$$\Hb^{2}=\{z\in \C:\: \Im{z}>0\}$$ of hyperbolic geometry. The space $\Sc_{k}$ is endowed with
the Petersson inner product $$\left\langle f,g \right\rangle_{PS} = \int\limits_{\Dc}f(z)\cdot \overline{g(z)} y^{k}\frac{dxdy}{y^{2}},$$ where $z=x+iy\in\Hb^{2}$, and $\frac{dxdy}{y^{2}}$ is the standard hyperbolic measure on $\Hb^{2}$.

Given $k$, let $\Bc_{k}=\{f_{1}=f_{k;1},\ldots, f_{N}=f_{k;N}\}$ be an orthonormal basis of $\Sc_{k}$ with respect to $\langle\cdot,\cdot\rangle_{PS}$. A random weight-$k$ cusp form is a random analytic functions $g_{k}:\Dc\rightarrow \C$
defined as
\begin{equation}
\label{eq:gk rand cusp form}
g_{k}(z) = \sum\limits_{j=1}^{N}a_{j}f_{j}(z),
\end{equation}
with the coefficients $(a_{j})_{j\le N} \in\C^{N}$ drawn uniformly on the unit sphere
\begin{equation}
\label{eq:sphere 2N-1}
\Scc^{2N-1}_{\C} = \left\{(z_{1},\ldots,z_{N})\in\C^{N}:\: \sum\limits_{j=1}^{N}|z_{j}|^{2} = 1 \right\},
\end{equation}
of real dimension $2N-1$. The law of
\eqref{eq:gk rand cusp form}
is independent of the choice of the orthonormal basis $\Bc_{k}$ of $\Sc_{k}$, since choosing a different orthonormal basis amounts to performing an orthogonal transformation on
the vector $(a_{j})_{j}$, under which the uniform measure of $\Scc^{2N-1}_{\C}$ is invariant.

\subsection{Statement of the principal result on compacta}

Let $\Kc\subseteq\Dc$ be a compact subdomain of positive area, that will be assumed to be fixed. We address the distribution of the random variable
\begin{equation}
\label{eq:Mk def sup}
\M_{k}=\M_{\Kc;k} = \sup\limits_{z\in\Kc}y^{k/2}\cdot \left|g_{k}(z)\right|
\end{equation}
for large $k$.
Our first principal result asserts that both the expectation and the median of $\M_{k}$ grow at the rate $\sqrt{\log{k}}$, and the distribution of $\M_{k}$ exponentially concentrates around the median:

\begin{theorem}
\label{thm:mean mag exp conc}
Let $g_{k}$ be a random weight-$k$ cusp form \eqref{eq:gk rand cusp form}, $\Kc\subseteq\Dc$ be a fixed compact subdomain of $\Dc$ of positive area, and $\M_{k}$ the random variable \eqref{eq:Mk def sup}.

\begin{enumerate}[i.]

\item The expectation of $\M_{k}$ is
\begin{equation*}
\E[\M_{k}] \asymp \sqrt{\log{k}},
\end{equation*}
i.e. there exist constants $c_{1}>c_{0}>0$, depending only on $\Kc$, so that
\begin{equation}
\label{eq:exp sup comm sqrt(logk)}
c_{0}\cdot \sqrt{\log{k}} \le \E[\M_{k}] \le c_{1}\cdot \sqrt{\log{k}}.
\end{equation}

\item Let $\mu_{k}=\mu(\M_{k})$ be the median of $\M_{k}$. Then
\begin{equation}
\label{eq:mean_median_diff}
\mu_{k} = \E[\M_{k}] + O_{\Kc}(1),
\end{equation}
and there exists some constant $c=c(\Kc)>0$ so that for every $r>0$ one has
\begin{equation}
\label{eq:exp conc}
\prob\left( \left|\M_{k}-\mu_{k}\right|>r\right) \le 2e^{-cr^{2}}.
\end{equation}

\end{enumerate}

\end{theorem}

Theorem \ref{thm:mean mag exp conc}(ii.) shows that the growth rate of the median $\mu(\M_{k})$ of $\M_{k}$ is of order of magnitude $\asymp \sqrt{\log{k}}$, but, just like the expectation of $\M_{k}$, divided by $\sqrt{\log{k}}$ it might fluctuate between
two constants, hence we could not replace $\mu_{k}$ in the exponential concentration \eqref{eq:exp conc} with an analytic expression of $k$. Instead, a weaker statement, that is a direct consequence of
Theorem \ref{thm:mean mag exp conc}, is asserted:

\begin{corollary}
\label{cor:exp decay prob}
There exist constants $c_{1}>c_{0}>0$ and $c>0$, depending only on $\Kc$, so that for every $r>0$ one has
\begin{equation*}
\prob\left( \M_{k}>c_{1}\cdot \sqrt{\log{k}}+r  \right) \le 2e^{-cr^{2}}
\end{equation*}
and
\begin{equation*}
\prob\left( \M_{k} < c_{0}\cdot \sqrt{\log{k}}-r  \right)\le 2e^{-cr^{2}}.
\end{equation*}
\end{corollary}

\subsection{Statement of the principal result for the global supremum}

We note that if $f$ is a modular form, defined on $\Hb^{2}$, then the function $$|y^{k/2}\cdot f(z)|$$ is invariant with respect to the action of $\tmop{SL}_2(\mathbb Z)$ on $\Hb^{2}$.
Our second principal result asserts upper
and lower bounds for the expected supremum of $y^{k/2}\cdot \left|g_{k}(z)\right|$ on $\Dc$ (or, what is equivalent in light of the said invariance, on the whole of $\Hb^{2}$).
Additionally, the distribution of the supremum exponentially concentrates around the median:

\begin{theorem}
\label{thm:sup bnd glob}
Let
\begin{equation*}
\M^{g}_{k}=\sup\limits_{z\in\Dc}y^{k/2}\cdot \left|g_{k}(z)\right|,
\end{equation*}
where `g' stands for ``global".
Then
\begin{enumerate}[i.]

\item We have
\begin{equation}
\label{eq:glob exp sup <<>>k^1/4}
k^{1/4} \ll \E[\M^{g}_{k}]\ll k^{1/4} \sqrt{\log{k}},
\end{equation}
where the constants involved in the ``$\ll$''-notation are absolute.

\item Let $\mu^{g}_{k}=\mu(\M^{g}_{k})$ be the median of $\M^{g}_{k}$. Then
\begin{equation}
\label{eq:mean_median_diff_glob}
\mu^{g}_{k} = \E[\M^{g}_{k}] + O(k^{1/4}),
\end{equation}
and there exists a constant $c>0$ so that for every $r>0$ one has
\begin{equation}
\label{eq:exp conc glob}
\prob\left( \left|\M^{g}_{k}-\mu^{g}_{k}\right|>r \right) \le 2e^{-\frac{cr^{2}}{k^{1/2}}}.
\end{equation}

\end{enumerate}

\end{theorem}

As a direct consequence, we obtain:

\begin{corollary}
\label{cor:exp decay prob glob}
There exist absolute constants $c_{1}>0$ and $c>0$ so that for every $r>0$ one has
\begin{equation*}
\prob\left( \M^{g}_{k}> c_{1}\cdot k^{1/4}\sqrt{\log{k}}+r \right) \le 2e^{-\frac{cr^{2}}{k^{1/2}}}.
\end{equation*}

\end{corollary}

\subsection{Conventions}

Throughout the manuscript we adopt the following conventions:

\begin{enumerate}[$\bullet$]

\item Given two positive quantities (or functions of some variable) $A,B$, we use the notation $A\ll B$ or, interchangeably, $A=O(B)$ if there exists some constant $C>0$ so that $|A|\le C\cdot B$.
If the constant $C$ depends on a parameter $P$, this may be designated $A\ll_{P} B$. Similarly, one writes $A\gg B$ (resp. $A\gg_{P} B$) if $B\ll A$ (resp. $B\ll_{P} A$), and $A\asymp B$ if both $A\ll B$ and $A\gg B$.

\vspace{2mm}

\item For a complex number $z=x+iy\in\Hb^{2}$ we designate its real and imaginary parts as $x=\Re{z}$ and $y=\Im{z}$ respectively.

\vspace{2mm}

\item We will reserve the low case letters for designating the real or complex valued
``spherical'' random fields (i.e. the coefficients are uniformly randomly drawn on a sphere), e.g. $g_{k}$ or $h_{k}$.

\vspace{2mm}

\item For two random variables (or random fields) $X,Y$, not necessarily defined on the same probability space, $X \stackrel{\Lc}{=} Y$ will mean equality {\em in law}.

\end{enumerate}

\subsection{Outline of the paper}

A discussion of the results of this manuscript, their relative standing within the existing literature on the subject, and their proofs, is conducted in \S~\ref{sec:discussion}.
The analysis of the covariance kernel (also called covariance function) of $g_{k}$, in different regimes,
on which the proofs of the principal results rest, is given in \S~\ref{sec:covar est}. The proofs of the exponential concentration results of Theorem \ref{thm:mean mag exp conc}(ii.) and Theorem \ref{thm:sup bnd glob}(ii.) 
will be given in \S~\ref{sec:low bnd conc}. The estimates on the expected supremum of Theorem \ref{thm:mean mag exp conc}(i.) and Theorem \ref{thm:sup bnd glob}(i.) will be established in \S~\ref{sec:proofs mean}.

\subsection{Acknowledgements}

A significant part of the research leading to this manuscript was conducted at the Mittag-Leffler Institute during the thematic semester in analytic number theory, $2024$,
and the authors are grateful to the Institute for its hospitality. I.W. would like to acknowledge the hospitality of Shandong University during his stay, also positively contributing to the
presented research.
The authors wish to thank Alon Nishry for freely sharing his expertise on the subject.
N.Y. was supported by the Israel Science Foundation (Grant No. 1881/20).

\section{Discussion}
\label{sec:discussion}

\subsection{Background}

The study of the sup norm problem for cusp forms on $\tmop{SL}_2(\mathbb Z)$ with large weight was initiated by Rudnick
\cite[Proposition A.1]{Rudnick}
who proved that the supremum of $f \in \Sc_k$ restricted to a compact subset $\Kc \subseteq \Dc$ satisfies
\begin{equation}
\label{eq:RudnickBoundIntro}
\sup_{z \in \Kc} y^{k/2} |f(z) | \ll_{\Kc} k^{1/2} \lVert f \rVert_2,
\end{equation}
where $\lVert f \rVert_2=\sqrt{\langle f,f\rangle_{PS}}$.
Over the entire domain ($\Dc$, or, what is equivalent, $\Hb^{2}$), Steiner \cite[Corollary 1.4]{Steiner} proved that for $f \in \Sc_k$,
\begin{equation}
\label{eq:SteinerBoundIntro}
 \sup_{z \in \Hb^{2}} y^{k/2} |f(z)| \ll k^{3/4} \lVert f \rVert_2.
\end{equation}
These bounds are best possible for $f \in \Sc_k$, since given any fixed, compact $\Kc \subseteq \Dc$ it follows from an argument of Sarnak
\cite[p. 2-3]{sarnak-2004}, using also Proposition \ref{prop:Fk covar}(ii.), that there exists a cusp form $f$ with
\[
\sup_{z\in \mathcal K} y^{k/2}|f(z)| \gg_{\mathcal K} k^{1/2} \lVert f \rVert_2 .
\]
Near the cusp there are even larger values and Sarnak's argument, using also
\cite[Theorem 1.3]{Steiner},
shows that there exists a cusp form $f$ with
\[
\sup_{z \in \Hb^{2}} y^{k/2}|f(z)| \gg k^{3/4} \lVert f \rVert_2 .
\]

\vspace{2mm}

For a Hecke cusp form $f$ and a
fixed compact subset $\Kc\subseteq\Dc$, one might predict that
\begin{equation}
\label{eq:bound ansatz}
\sup_{z \in \Kc} y^{k/2} |f(z)| \ll_{\Kc} k^{\varepsilon}
\end{equation}
for any $\varepsilon>0$.
Xia \cite{xia-2007} showed that if $f$ is a Hecke cusp form, then
\begin{equation}
\label{eq:Xia <<}
 \sup_{z \in \Dc} y^{k/2} |f(z)| \ll k^{1/4+\varepsilon} \lVert f \rVert_2
\end{equation}
for any $\varepsilon>0$. That is nearly optimal, since Xia also proved for any Hecke cusp form $f$ that
\begin{equation}
\label{eq:Xia >>}
 \sup_{z \in \Dc} y^{k/2} |f(z)| \gg k^{1/4-\varepsilon} \lVert f \rVert_2
\end{equation}
for any $\varepsilon>0$, where the maximal value of $f$ is attained near the cusp at $\infty$, more precisely, at $y\approx \frac{k}{4\pi}$.
These global sup norm results for Hecke cusp forms are roughly comparable to the global sup norm for random cusp forms given in Theorem \ref{thm:sup bnd glob}.

% \vspace{2mm}

% In\fixme{SL: should we remove this paragraph? it seems a bit off topic to me.} the analogous setting of Hecke-Maass forms on $\tmop{SL}_2(\mathbb Z)$, it follows from work of Mili\'{c}evi\'{c} \cite[Theorem 1]{milicevic-2010} that for any fixed compact set
% $\Kc \subseteq \Dc$ there exist a subsequence $\{ \phi_j \}$ of Hecke-Maass forms such that
% \[
% \sup_{z \in \Kc} |\phi_j(z)| \gg_{\Kc} \exp\bigg(\sqrt{\frac{\log \lambda_j}{\log \log \lambda_j}}(1-o(1)) \bigg)
% \]
% where $\lambda_j$ is the Laplace-Beltrami eigenvalue of $\phi_j$. However, for a ``generic" Fuchsian subgroup of the first kind $\tmop{\Gamma} \subseteq \tmop{SL}_2(\mathbb R)$, Berry's random wave model \cite{berry-1977}
% predicts that for compact $\Kc \subseteq \tmop{\Gamma} \backslash \mathbb H $, one has
% \[\sup_{z \in \mathcal K} |\phi_j(z)| \ll_{\mathcal K} \sqrt{\log \lambda_j},\]  see \cite[Eq'n (1.10)]{sarnak-1993}.

\subsection{Discussion of results}

\subsubsection{{\bf Almost sure bounds for the supremum}}

A straightforward application of the Borel-Cantelli lemma with the exponential probability tail decay of Corollary \ref{cor:exp decay prob} implies that the supremum of $g_{k}$ on compacta is {\em almost surely}
$O\left(\sqrt{\log{k}} \right)$, i.e. for every compact $\Kc\subseteq\Dc$ there exists a constant $C=C(\Kc)$ so that the inequality $$\sup_{z \in \Kc} y^{k/2} |g_{k}(z)|\le C\cdot \sqrt{\log{k}}$$ holds with probability $1$,
assuming that the $g_{k}(\cdot)$ are drawn independently
for different $k$. Hence, the presented results are in strong support to the heuristic ansatz \eqref{eq:bound ansatz}.

For the global supremum, one may analogously infer the almost sure bound $$\M^{g}_{k} = O\left(k^{1/4}\sqrt{\log{k}} \right)$$ from Corollary \ref{cor:exp decay prob glob}.
Our analysis below shows that the {\em global} supremum of $y^{k/2} |g_{k}(z)|$ is attained, with high probability,
around $\Im{z}\approx \frac{k}{4\pi}$,
consistent to Xia's more restricted result \eqref{eq:Xia <<}-\eqref{eq:Xia >>}, both in terms of the order of {\em magnitude} of the supremum,
and the {\em location} where it is attained.

\subsubsection{{\bf Spherical vs. Gaussian coefficients}}
\label{sec:spher vs Gauss}

Since the central object of study in this manuscript is the supremum of cusp forms across different regimes, it is natural to focus on cusp forms restricted to have unit $L^2$-Petersson norm. That is,
the coefficients $(a_{j})_{j\le N}$ are constrained to lie on the unit sphere \eqref{eq:sphere 2N-1} of real dimension $2N-1$. Accordingly, the model \eqref{eq:gk rand cusp form} of cusp forms where 
the coefficients $(a_{j})$ are uniformly distributed
on $\Scc^{2N-1}_{\C}$ is the most immediate and natural choice.
Instead, one might study the supremum of  
\begin{equation}
\label{eq:Gk Gauss def}
G_{k}(z) = \frac{1}{\sqrt{N}}\sum\limits_{j=1}^{N}b_{j}f_{j}(z),
\end{equation}
where the $b_{j}$ are standard complex Gaussian i.i.d., or rather the supremum of $y^{k/2}|G_{k}(z)|$. The covariance kernel of $G_{k}(\cdot)$, coinciding with the covariance kernel of $g_{k}(\cdot)$, determines the Gaussian random field $G_{k}(\cdot)$
(though not a non-Gaussian random field), hence, in principle, it is possible to recover any property of $G_{k}(\cdot)$ solely in terms of its covariance.

One has the equality in law:
\begin{equation}
\label{eq:G_{k}=zeta gk}
G_{k}(z)\stackrel{\Lc}{=}\zeta \cdot g_{k}(z),
\end{equation}
where $\zeta=\zeta_{N}>0$ is a random variable, independent of $g_{k}(\cdot)$, 
distributed according to the $\chi$ distribution with $2N$ degrees of freedom. For large $N$, the distribution of $\zeta$ is highly concentrated at $1$, and
its mean is given precisely by
\begin{equation}
\label{eq:mean zeta}
\E[\zeta] = \frac{1}{\sqrt{N}}\cdot\frac{\Gamma(N+1/2)}{\Gamma(N)}=1+O\left(\frac{1}{N}\right),
\end{equation}
see e.g. ~\cite[p. 238]{Stats with Math}. It is therefore easy to infer the results for the ``spherical" random fields from their Gaussian counterparts (and vice versa). 

\vspace{2mm}

The Gaussian approach will be invoked (in a somewhat primitive form) for the purpose of proving the lower bound for the expected global supremum of Theorem \ref{thm:sup bnd glob}(i.).
In fact, one way to establish the upper bounds for the expected suprema of Theorem \ref{thm:mean mag exp conc}(i.) 
and Theorem \ref{thm:sup bnd glob}(i.) is via Dudley's {\em entropy} method
(see e.g. ~\cite[\S 1.3]{AdTa}), which is applicable to Gaussian (real-valued) random fields and is not pursued in this manuscript.

\vspace{2mm}

The same results hold for the real-valued coefficients $a_{j}$ (spherical or Gaussian), via the same analysis that mainly appeals to the covariance kernel restricted to the diagonal.

\subsubsection{{\bf Random sections of tensor powers of line bundles on K\"{a}hler manifolds}}

The model \eqref{eq:gk rand cusp form} of random cusp forms with i.i.d. coefficients
is a particular ensemble falling under the scope addressed in the literature on random sections of tensor powers of line bundles
on K\"{a}hler manifolds, compact and non-compact, see ~\cite[\S~4]{DrLiMar sup}. So far, the main focus of that line of research has been on the zeros of the said sections, in particular their $n$-point correlations,
see e.g. ~\cite{BlShZe Inv,ShZe GAFA} for compact manifolds or ~\cite{MaMar Adv,DrLiMar Gauss,DrLiMar sup} for non-compact manifolds, the references within, and their followups.
Though, formally, the said results on the zeros of random sections are not directly applicable on the model \eqref{eq:gk rand cusp form} of random cusp forms (rather, the Gaussian model),
their ``perturbative" techniques seem to directly apply here, possibly in conjunction with the asymptotics for the covariance function of ~\S~\ref{sec:covar est} below, to yield the analogous results on the correlations of zeros of
$g_{k}(\cdot)$.

To the best knowledge of the authors of this manuscript, the only result in the literature pertaining to the supremum of random sections is \cite[Theorem 1.4]{DrLiMar sup}, which asserts upper and lower bounds for the expected supremum
of random sections, in a vastly general situation, both in terms of the manifolds and the distribution of the coefficients analogous to the $a_{j}$ in \eqref{eq:gk rand cusp form}. In particular, their result
~\cite[Corollary 4.1]{DrLiMar sup} is applicable to the model \eqref{eq:gk rand cusp form}, yielding upper and lower bounds for the expected supremum in Theorem \ref{thm:mean mag exp conc}(i.), in terms of
non-matching powers of $k$. It is likely that the {\em optimal} bounds in \eqref{eq:exp sup comm sqrt(logk)} could be extended beyond the spherical (equivalently, Gaussian) coefficients, for some class of non-Gaussian distributions,
by using the Central Limit Theorem, and the related techniques in probability theory. We leave this, the above, and other associated questions to be addressed elsewhere.

\subsubsection{{\bf Suprema of random waves on compact manifolds}}

Burq and Lebeau ~\cite{BurqLebeau} considered the supremum of {\em random waves} on {\em compact} manifolds, i.e. linear combinations of Laplace eigenfunctions belonging to an energy window corresponding to an energy $\lambda$,
in the high energy limit $\lambda\rightarrow\infty$. They proved the analogues ~\cite[Theorem 5]{BurqLebeau} of the results of Theorem \ref{thm:mean mag exp conc}, with the expected supremum also scaling logarithmically as $\sqrt{\log{\lambda}}$.
The proofs of Theorem \ref{thm:mean mag exp conc} and parts of Theorem \ref{thm:sup bnd glob} are inspired by the techniques presented in ~\cite{BurqLebeau}, though having marked obstacles and novel ingredients, in particular, of arithmetic nature.
The exponential concentration of the supremum in ~\cite[Theorem 5]{BurqLebeau} has been recently refined by \cite[Theorem 2.3]{AdFeYa}.

For both \cite{BurqLebeau} and this manuscript, the asymptotic analysis of the covariance kernel of the random ensemble, and, in particular, its restriction to the diagonal, is instrumental.
For the upper bound in \eqref{eq:exp sup comm sqrt(logk)}, and, a forteriori the matching lower bound, our methods allow for the domain $\Kc$ to polynomially expand with $k$. Specifically, it is possible to 
establish the analogue of \eqref{eq:exp sup comm sqrt(logk)}, where $\M_{k}$ is the supremum of $y^{k/2}|g_{k}(z)|$ on $$\Dc\cap \left\{z\in \Dc:\: \Im{z} < C\cdot \sqrt{k}\right\},$$ with $C>0$ an arbitrary constant,
instead of a fixed compact domain $\Kc$.

\subsubsection{Analysis of covariance kernel}

An asymptotic treatment of the covariance kernel not constrained to the diagonal is also included, see Theorem \ref{thm:bergman}, and \eqref{eq:relationrr}. 
We believe it to be of independent interest, also allowing for some future applications, for example, for the purpose of a further deep analysis of the distribution of the supremum $\M_{k}$.

\vspace{2mm}

Our arguments show that the bulk of the contribution for the global supremum of Theorem \ref{thm:sup bnd glob} comes from a rectangle in $\Dc$ at height $\approx \frac{k}{4\pi}$, corresponding
to the maximal variance of $y^{k/2}g_{k}(z)$, see the proof of the lower bound of Theorem \ref{thm:sup bnd glob}(i.) in \S~\ref{sec:proofs mean}.
A somewhat heuristic analysis of the covariance on that ``island of maximal variance", of length $1$ and width $\approx \sqrt{k}$, suggests an upper bound of $k^{1/4}$ for 
the expected supremum of $y^{k/2}g_{k}(z)$ on that island, and, possibly the {\em global} expected supremum. Thus, it is plausible that the lower bound in \eqref{eq:glob exp sup <<>>k^1/4} 
is, in fact, the true order of magnitude of $\E[\M_{k}^{g}]$.

\subsection{On the proofs of the main results}

\subsubsection{{\bf Asymptotic analysis of the covariance kernel}}

Let $r_{k}(z,w)=\E[h_{k}(x)\cdot \overline{h_{k}(w)}]$ be the covariance kernel of $h_{k}(z):=y^{k/2}g_{k}(z)$. Though $r_{k}(\cdot,\cdot)$ does not determine the law of $h_{k}$,
it does so in the Gaussian case, which is intimately related to $h_{k}$ (see the discussion in \S~\ref{sec:spher vs Gauss} above).
The asymptotic analysis of $r_{k}(\cdot,\cdot)$ in different regimes, performed within \S~\ref{sec:covar est},
is central to the proofs of the main results. Our argument begins by relating $r_k(\cdot,\cdot)$ to the Bergman kernel, see \eqref{eq:relationrr}. The Bergman kernel has been widely studied and we defer discussion of the literature and previous results on this topic to \S~\ref{sec:bergman}.
Returning back to the proof, for $z,w \in \Dc$ not too close modulo $\tmop{SL}_2(\mathbb Z)$ Theorem \ref{thm:bergman} shows that the Bergman kernel is small, implying $r_k(z,w)$ is small for such $z,w$. Next, for $z,w \in \Dc$ which are close together and not too close to the cusp at infinity, we show that in the sum over $\gamma \in \tmop{SL}_2(\mathbb Z)$ defining the Bergman kernel \eqref{eq:bergmandef} the contribution from the terms corresponding to $\gamma=\pm I$ dominates unless $z$ is nearby an elliptic point in which case there are additional terms arising from the stabilizer group of the elliptic point that must be considered.

Near the cusp at infinity a more delicate analysis of the Bergman kernel is needed and for simplicity we only consider the diagonal $z=w$. For $z\in \Dc$ near the cusp at infinity, the main contribution to the sum defining the Bergman kernel \eqref{eq:bergmandef} arises from the terms corresponding to the stabilizer group of infinity, i.e. the matrices $\gamma=\begin{pmatrix} \pm 1 & n \\ 0 & \pm 1 \end{pmatrix}$ where $n \in \mathbb Z$. We evaluate the sum over these matrices using the Poisson summation formula in Lemma \ref{lem:transform} and then analyze the resulting sum using Laplace's method. This leads to precise upper and lower bounds for the Bergman kernel valid near the cusp at infinity, which are given in Theorem \ref{thm:nearthecusp}.

\subsubsection{{\bf \texorpdfstring{$L^p$}{Lp} norm asymptotics and concentration}}

The proof of Theorem \ref{thm:mean mag exp conc} is inspired by the techniques of Burq-Lebeau \cite{BurqLebeau}. We use the sup norm bound \eqref{eq:RudnickBoundIntro} together with a standard concentration inequality on the sphere (L\'evy's inequality), to show that the expectation of the sup norm is close to its median $\mu_{k}$. The difference between the median $\mu_k$ and the expected value of the sup norm is then bounded using this result, and this argument 
establishes Theorem \ref{thm:mean mag exp conc}(ii.).  Upon using the sup norm bound \eqref{eq:SteinerBoundIntro} in the global case, the concentration in Theorem \ref{thm:sup bnd glob}(ii.) follows similarly.

Denoting the $L^p$ norm by $\left\Vert h_{k}\right\Vert _{p,\Kc}$, the median $\mu_{k,p}$ is shown to be close to $\left(\mathbb{E}\left[\left\Vert h_{k}\right\Vert _{p,\Kc}^{p}\right]\right)^{1/p}$ which admits an exact formula in terms of the variance $r_k(z,z)$.  This analysis gives an asymptotic formula for the expected value of $\left\Vert h_{k}\right\Vert _{p,\Kc}$, which is given Theorem \ref{thm:expectationlpnorm}(i.). By a similar argument we obtain asymptotics for the expected value of $\left\Vert h_{k}\right\Vert _{p,\Dc}$, which is given in Theorem \ref{thm:expectationlpnorm}(ii.).

In \S~\ref{sec:proofs mean} we relate the $L^p$ norm of $h_k$ to its sup norm. Combining the $L^p$ norm estimates given in Theorem \ref{thm:expectationlpnorm} with the estimates for $r_k(z,z)$ in \S~\ref{sec:covar est} and choosing $p$ to be proportional to $\log k$ we obtain the upper bounds for the expected value of the sup norm of $h_k$ stated in Theorem \ref{thm:mean mag exp conc}(i.) and Theorem \ref{thm:sup bnd glob}(i.). The lower bound in Theorem \ref{thm:mean mag exp conc}(i.) follows from a similar analysis. 
% However our arguments barely fail to yield lower bounds for $L^p$ norm with $p$ proportional to $\log k$, so that another approach is needed.

% given Theorem \ref{thm:expectationlpnorm}(i)  In  
% the $L^p$ norm 
% Upon choosing $p$ proportional to $\log k$, this analysis gives the desired lower bound of Theorem \ref{thm:mean mag exp conc}.

\vspace{2mm}

While we were not able to extend this approach for the lower bound to the global case, the lower bound of Theorem \ref{thm:sup bnd glob}(i.) follows from a different argument: it is sufficient to find a single point $z_{k}$ so that the variance $r_{k}(z_{k},z_{k})$ of $h_{k}(z)=y^{k/2}\cdot g_{k}(z)$ at $z=z_{k}$ is of order of magnitude $k^{1/2}$. From the fact that for a (real or complex) Gaussian random variable $Z$
it follows that $$\E\left[|Z|\right] \gg \var(Z)^{1/2},$$ we will be able to infer the same for the (non-Gaussian) random variable $h_{k}(z_{k})$ via its Gaussianity connection explained in \S~\ref{sec:spher vs Gauss} above. 
That will give the lower bound $$\E[|h_{k}(z_{k})|] \gg k^{1/4} , $$ and hence the same for the sup norm of $h_k$ on $\Dc$.

\section{Analysis of covariance kernels}
\label{sec:covar est}

A key ingredient of our techniques is the asymptotic analysis of the covariance kernel
\begin{equation}
\label{eq:rk covar def}
r_{k}(z,w) = \E\left[h_{k}(z) \cdot \overline{h_{k}(w)} \right]
\end{equation}
$z=x+iy,\, w=u+iv \in \Hb^{2}$, of the random field
\begin{equation}
\label{eq:fk spher field def}
h_{k}(z):=y^{k/2}g_{k}(z),
\end{equation}
in different regimes, corresponding to Theorem \ref{thm:mean mag exp conc} and Theorem \ref{thm:sup bnd glob} respectively.
Since it is easy to check that the $a_{j}$ are uncorrelated, and $\E[|a_{j}|^{2}] =\frac{1}{N}$ for $j\le N$, the covariance kernel is given by
\begin{equation} \label{eq:rkdef}
r_{k}(z,w) = \frac{1}{N}\sum\limits_{j=1}^{N}y^{k/2}f_{j}(z)\cdot v^{k/2}\overline{f_{j}(w)}.
\end{equation}
Just like the law of $g_{k}(\cdot)$, the law of $h_{k}(\cdot)$, and, in particular, its covariance kernel, are independent of the choice of the orthonormal basis $\Bc_{k}$ of $\Sc_{k}$.

%In particular, the variance of $f_{k}(\cdot)$ is equal to
%$$\E\left[ \left| f_{k}(z) \right|^{2}\right] = r_{k}(z,z),$$ i.e. the covariance kernel evaluated at the diagonal.

\subsection{Statement of results: Asymptotics of covariance kernels}

The following proposition deals with the asymptotic behavior of the variance $r_{k}(z,z)=\E\left[\left|h_{k}(z)\right|^{2}\right]$, as defined in \eqref{eq:rk covar def} with $z=w$.
% \begin{equation}
% \label{eq:var fk(z)=r(z,z)}
% \E\left[ \left| h_{k}(z) \right|^{2}\right] = r_{k}(z,z).
% \end{equation}

\begin{proposition}[Variance within the bulk]
\label{prop:Fk covar}
% Let $F_{k}:\Dc\rightarrow\C$ the Gaussian random field
% \begin{equation}
% \label{eq:Fk def}
% F_{k}(z) = y^{k/2}g_{k}(z),
% \end{equation}
% where $z=x+iy$.
% \begin{enumerate}[i.]
% \item
% Let $\Dc_{k}\subseteq \Dc$ be the sequence of subdomains
% $$\Dc_{k}=\{z\in\Dc:\: \Im{z}\le k^{1/2}/(9 \log k)^{1/2}\}$$ of $\Dc$.
% The variance of the Gaussian random variables $h_{k}(\cdot)$ in \eqref{eq:var fk(z)=r(z,z)} is uniformly bounded on $z\in \Dc_{k}$,
% i.e. there exists a constant $C_{0}>0$ such that
% \begin{equation*}
% \sup\limits_{z\in \Dc_{k}}r_{k}(z,z) = \sup\limits_{z\in \Dc_{k}} \E\left[\left|h_{k}(z)\right|^{2}\right]\le C_0.
% \end{equation*}
% \item
For $0< \delta <1$ let $\eta_{\delta,{j}}$ be the neighborhoods of the elliptic points
\begin{equation}
\label{eq:elliptic points}
e_1=e^{\pi i/3}, e_2=i, \text{ and } e_3=e^{2\pi i/3}
\end{equation}
of $\tmop{SL}_2(\mathbb Z) \backslash \mathbb H^2$:
$$\eta_{\delta,j}=\{z\in \Dc: |z-e_j| \le \delta\},$$ $j=1,2,3$. Denote the subdomain
\begin{equation}
\label{eq:Fdeltadef}
\Fc_{\delta}
=\Dc\setminus \{ \eta_{\delta,1} \cup \eta_{\delta,2} \cup \eta_{\delta,3} \}
\end{equation}
of the canonical fundamental domain $\Dc$.
Then there exists $c_0>0$ such that, uniformly on $z \in \Fc_{\delta}$, it holds that
\begin{equation*}
\E\left[\left|h_{k}(z)\right|^{2}\right] = \frac{k-1}{4\pi N}+O\left( e^{-c_0\delta^2 k}+y e^{-k/(17y^2)}   \right),
\end{equation*}
where $h_{k}(\cdot)$ is as in \eqref{eq:fk spher field def}.
Equivalently, one has uniformly on $z \in \Fc_{\delta}$
\begin{equation}
\label{eq:var asuymp Dk sum}
y^{k}\sum_{j=1}^{N}\left|f_{j}\left(z\right)\right|^{2}=\frac{k-1}{4\pi}+O\left( ke^{-c_0\delta^2 k}+yk e^{-k/(17y^2)}   \right) .
\end{equation}

% \end{enumerate}

\end{proposition}

% and the boundedness of the variance of Proposition \ref{prop:Fk covar}.

% Before stating the next result let us introduce the following notation. Given $z,w \in \mathbb H^2$, let
% \begin{equation}
% u(z,w)=\frac{|z-w|^2}{4 \Im z \Im w},
% \end{equation}
% which is related to the hyperbolic distance $d_{\mathbb H^2}(z,w)$ by the formula $\cosh d_{\mathbb H^2}(z,w)=2u(z,w)+1$ (see \cite[Equation (1.3)]{Iwaniec-2002}).

The next result provides nearly optimal bounds
for the variance $r_k(z,z)$ throughout $\Dc$.

\begin{proposition}[Uniform estimates for the variance] \label{prop:nearthecusp}
Let $z=x+iy, w=u+iv \in \Dc$. Each of the following hold.
\begin{enumerate}[i.]
\item Uniformly, we have for any $A \ge 1$ that
\[
\E\left[ \left| h_{k}(z) \right|^{2}\right]=r_k(z,z) \ll_A \begin{cases}
1+\frac{y}{\sqrt{k}} & \text{ if }
\min_{n \in \mathbb Z} | n-\frac{k-1}{4\pi y}| \le \frac{\sqrt{k} \log k}{y}, \\
k^{-A} & \text{ otherwise.}
\end{cases}
\]
\item If $y \ge k/(2\pi)$ then \[\E\left[ \left| h_{k}(z) \right|^{2}\right]=r_k(z,z) \ll e^{-k/37}.\]
\item We have
\[
\E\left[ \left| h_{k}(z) \right|^{2}\right]=r_k(z,z)\gg \displaystyle\begin{cases}
1 & \text{ if }  2 \le y \le \frac{\sqrt{k}}{12\pi}, \\
\frac{y}{\sqrt{k}} & \text{ if }  \frac{\sqrt{k}}{12\pi} < y < \frac{k}{2\pi}  \text{ and } \min_{n \in \mathbb N } \big|\frac{k-1}{4\pi y}-n\big| \le \frac{\sqrt{k-1}}{12 \pi y}.
\end{cases}
\]
\end{enumerate}
\end{proposition}

\subsection{Analysis on the bulk: Proof of Proposition \ref{prop:Fk covar}} \label{sec:bergman}

%\subsection{Proof of Proposition \ref{prop:Fk covar}}

In what follows we are going to state an asymptotic result on the covariance function $r_{k}(z,w)$ in \eqref{eq:rk covar def}, not restricted to the diagonal, which easily implies Proposition \ref{prop:Fk covar}.
To this end, we will establish precise estimates for the Bergman kernel, which we introduce next.

For $z,w \in \Hb^{2}$ and
$\gamma=\begin{pmatrix}
a & b \\
c & d
\end{pmatrix}$ define
\[
b_{\gamma}(z,w)=\frac{2i}{w+\gamma z} \cdot\frac{1}{cz+d}.
\]
For even $k \ge 4$, given $z=x+iy,w=u+iv \in \Hb^{2}$, the Bergman kernel for weight $k$ cusp forms on $\tmop{SL}_2(\mathbb Z)$ is
\[
B_k(z,w):=\sum_{\gamma \in \tmop{SL}_2(\mathbb Z)} b_{\gamma}(z,w)^k.
\]
The function $B_k(\cdot,\cdot)$ is holomorphic in both variables. It is also a reproducing kernel, that is, for $f \in \Sc_k$
\[
\int_{\tmop{SL}_2(\mathbb Z) \backslash\mathbb H} y^k f(z) \overline{B_k(z,-\overline w)}  \frac{dx \, dy}{y^2}=\frac{8\pi}{k-1} f(w)
\]
see \cite[Theorem 2.15]{Steiner} as well as \cite[Section 2, Proposition 1]{zagier-1976}, \cite{zagier-1977}. Let us write
\begin{equation} \label{eq:bergmandef}
R_k(z,w):=(yv)^{k/2} B_k(z,-\overline w)=\sum_{\gamma \in \tmop{SL}_2(\mathbb Z)}\ell_{\gamma}(z,w)^k,
\end{equation}
where $\ell_{\gamma}(z,w)=\sqrt{yv} b_{\gamma}(z,-\overline w)$.  By \cite[Theorem 2.15]{Steiner} (see also the first line of the proof of \cite[Corollary 2.16]{Steiner}), we have for $z,w\in \mathbb H^2$ that
\begin{equation} \label{eq:reproducing}
\sum_{j=1}^N y^{k/2} f_j(z) \overline{ v^{k/2} f_j(w)} = \frac{k-1}{4\pi} \cdot \frac{R_k(z,w)}{2}.
\end{equation}
Hence, this yields the following formula relating the Bergman kernel to the covariance kernel
\begin{equation} \label{eq:relationrr}
r_{k}(z,w)=\frac{k-1}{4\pi N} \cdot \frac{R_k(z,w)}{2}=\bigg(\frac{3}{\pi}+O\bigg(\frac1k \bigg)\bigg) \frac{R_k(z,w)}{2}
\end{equation}
of $h_{k}(\cdot)$, for $z,w \in \mathbb H^2$.
Estimates for $R_k(z,z)$ in the context of weight $k$ modular forms for certain Fuchsian subgroups of the first kind
have been given in \cite{abms, am,michel-ullmo-1998,jorgenson-kramer-2004, jorgenson-kramer-2011, friedman-jorgenson-kramer-2016,Steiner, ma-marinescu-2021}. In particular, Steiner \cite[Theorem 1.3]{Steiner} has proved
\[
\sup_{z \in \mathbb H^2} R_k(z,z) \asymp k^{1/2}.
\]
For further background and results on the Bergman kernel, such as results for compact manifolds, we refer the reader
to \cite{tian-1990,zelditch-1998,ma-marinescu-2007}.

\vspace{2mm}

With
the newly introduced notation we are finally able to state the aforementioned result on the Bergman kernel. We believe
Theorem \ref{thm:bergman} to be of independent interest, and of high potential for further applications.

\begin{theorem} \label{thm:bergman}
Let $k \ge 4$ be even.
For $\delta >0$ let $\Fc_{\delta}$ be the subdomain \eqref{eq:Fdeltadef} of $\Dc$. There exists an absolute constant $c_0>0$ such that
uniformly for all $0<\delta<1$, $z \in \mathcal F_{\delta}$, and $|z-w|\le c_0 \delta$ one has
\[
R_k(z,w)=2 \bigg(\frac{2i \sqrt{yv}}{z-\overline w}  \bigg)^k+O(e^{-c_0 \delta^2 k}+y e^{-k/(17y^2)}).
\]
\end{theorem}

\vspace{2mm}

A refinement of the subsequent argument yields a uniform estimate for $R_k(z,w)$ for all $z \in \Dc$ and $|z-w| \le c_0 \delta$, with additional terms that are non-negligible near the elliptic points. For example, for $|z-i|\le \delta$, $|z-w| \le c_0 \delta$, and $z \in \Dc$ one has
\[
R_k(z,w)=2 \bigg(\frac{2i \sqrt{yv}}{z-\overline w}  \bigg)^k+2\bigg(\frac{2i \sqrt{yv}}{z \overline w+1} \bigg)^k+O(e^{-c_0 \delta^2 k}),
\]
for some fixed, sufficiently small $c_0>0$; and if $k \equiv 2 \pmod 4$ and $z=w=i$ the main term vanishes,
which is consistent with the fact that $f(i)=0$ for all modular forms of weight $k \equiv 2 \pmod 4$.

\vspace{2mm}

The proof of Theorem \ref{thm:bergman} will be given in \S~\ref{sec:Bergman kernel} below. In the meantime, we give the announced proofs
for Proposition \ref{prop:Fk covar}.

\begin{proof}[Proof of Proposition \ref{prop:Fk covar} assuming Theorem \ref{thm:bergman}]
Applying \eqref{eq:reproducing} and
Theorem \ref{thm:bergman}(i.) with $z=w$ and noting $\tfrac{2i y}{z-\overline z}=1$ we get
\[
\mathbb E[ y^{k} |g_k(z)|^2]= \sum_{j=1}^N y^{k} |f_j(z)|^2 = \frac{k-1}{4\pi }+O(ke^{-c_0 \delta^2 k}+yk e^{-k/(17y^2)})
\]
for $z \in \mathcal F_{\delta}$ and $|z-w| \le c_0 \delta$. This establishes Proposition \ref{prop:Fk covar}. \end{proof}

\subsection{Asymptotic analysis of the Bergman kernel: Proof of Theorem \ref{thm:bergman}}
\label{sec:Bergman kernel}

\begin{proof}[Proof of Theorem \ref{thm:bergman}] The strategy of the proof broadly follows the approach of Cogdell and Luo \cite{cogdell-luo-2011}.
% Since we consider $\tmop{SL}_2(\mathbb Z)$ we are able to refine their techniques and prove a stronger result in this setting.
For $z,w \in \Hb^{2}$ let
$u(z,w)=\tfrac{|z-w|^2}{4\Im{z}\cdot \Im{w}}$, which is related to $d_{\mathbb H^2}(z,w)$ by the formula $\cosh d_{\mathbb H^2}(z,w)=2u(z,w)+1$ (see \cite[Equation (1.3)]{Iwaniec-2002}).
Given $\gamma \in \tmop{SL}_2(\mathbb Z)$
write $\gamma z=x'+iy'$. Observe that for $\gamma=\begin{pmatrix}
a & b \\
c & d
\end{pmatrix}$ we have $y'=\frac{y}{|cz+d|^2}$, so that
\begin{equation} \label{eq:distance}
\begin{split}
|\ell_{\gamma}(z,w)|=\frac{2\sqrt{vy' }}{|\gamma z-\overline w|}
=& \frac{2\sqrt{vy'}}{((u-x')^2+(v+y')^2-(v-y')^2+(v-y')^2)^{1/2}}\\
=& \frac{2\sqrt{vy'}}{(|\gamma z-w|^2+4vy')^{1/2}}
=(1+u(w,\gamma z))^{-1/2}.
\end{split}
\end{equation}
Next note that $\ell_{\pm I}(z,w)=\tfrac{2i\sqrt{yv}}{z-\overline w}$. Also,
we have
\begin{equation} \notag
\sum_{\substack{\gamma \in
\tmop{SL}_2(\mathbb Z) \\ \gamma \neq \pm I}} (1+u(w,\gamma z))^{-k/2} \le
\max_{\substack{\gamma \in \tmop{SL}_2(\mathbb Z) \\ \gamma\neq \pm I}}(1+u(w,\gamma z))^{-k/2+2}
\sum_{\substack{\gamma \in
\tmop{SL}_2(\mathbb Z) \\ }} (1+u(w,\gamma z))^{-2}.
\end{equation}
Thus, we have
\begin{equation} \label{eq:firststep}
R_k(z,w)=2 \bigg( \frac{2i\sqrt{yv}}{z-\overline w}\bigg)^k+O\bigg( \max_{\substack{\gamma \in \tmop{SL}_2(\mathbb Z) \\ \gamma\neq \pm I}}(1+u(w,\gamma z))^{-k/2+2}
\sum_{\substack{\gamma \in
\tmop{SL}_2(\mathbb Z) \\ }} (1+u(w,\gamma z))^{-2}   \bigg).
\end{equation}
To simplify the analysis below, we note that since $z,w \in \Dc$ and $|z-w| \le c_0 \delta$ we have
 $v/y =1+ O(c_0 \delta/y)$,
and we conclude that
\begin{equation} \label{eq:triangle}
u(w,\gamma z)=\bigg|\frac{\gamma z -z}{2\sqrt{y y'} \sqrt{\frac{v}{y}}}+\frac{z-w}{2\sqrt{vy'}} \bigg|^2
=u(z,\gamma z)+O\bigg(\frac{c_0\delta\sqrt{u(z,\gamma z)}}{\sqrt{yy'}}+\frac{c_0^2 \delta^2}{yy'}+\frac{c_0\delta u(z,\gamma z) }{y}\bigg).
\end{equation}

To bound the sum over $\gamma \in \tmop{SL}_2(\mathbb Z)$ in the error term in \eqref{eq:firststep} we will use \eqref{eq:triangle} along with the a uniform bound for the number of hyperbolic lattice points inside a hyperbolic circle. By \cite[Corollary 2.12]{Iwaniec-2002} we have for any $X \ge 1$
\begin{equation} \label{eq:iwaniec}
\# \{ \gamma \in \tmop{SL}_2(\mathbb Z) :
u(z,\gamma z) \le X \} \ll y X.
\end{equation}
% Alternatively, in our setting the preceding bound follows from adapting Gauss' argument for counting $\mathbb Z^2$-lattice points and also taking into account the distribution of hyperbolic lattice points near the cusp at $\infty$.
Using \eqref{eq:triangle} and \eqref{eq:iwaniec} we have
\begin{equation} \label{eq:sumbd}
\begin{split}
\sum_{\substack{\gamma \in
\tmop{SL}_2(\mathbb Z) \\ }} (1+u(w,\gamma z))^{-2} & \ll
\sum_{\substack{\gamma \in
\tmop{SL}_2(\mathbb Z) \\ }} (1+u(z,\gamma z))^{-2} \\
&
\ll y+ \sum_{n=0}^{\infty}  2^{-2n}
\sum_{\substack{\gamma \in
\tmop{SL}_2(\mathbb Z) \\ 2^n \le u(z,\gamma z) \le 2^{n+1}}} 1
\ll y+y \sum_{n=0}^{\infty} 2^{-n} \ll y.
\end{split}
\end{equation}

We will now consider separately the cases $y \ge 2$ and $y < 2$.
If $y \ge 2$ and $\gamma \neq \pm I$ then, using \eqref{eq:triangle} and that $c_0$ is sufficiently small, we have
$\tfrac{1}{8y^2} \le \tfrac12 u(z,\gamma z) \le u(w,\gamma z)$ and we get that
\begin{equation} \label{eq:ubd}
\max_{\substack{\gamma \in \tmop{SL}_2(\mathbb Z) \\ \gamma\neq \pm I}}(1+u(w,\gamma z))^{-k/2+2}
\le \bigg(1+\frac{1}{8y^2} \bigg)^{-k/2} \le e^{-k/(17y^2)}
\end{equation}
where we have used that $\log(1+t) \ge 63t/64 > 16t/17$ for $0\le t \le 1/32$ in the last step. Using \eqref{eq:sumbd} and \eqref{eq:ubd} in \eqref{eq:firststep} completes the proof in the case $y \ge 2$.

\vspace{2mm}

Next, consider the case $y <2$. By \eqref{eq:firststep}, \eqref{eq:triangle}, and \eqref{eq:sumbd} it suffices to show that if $z \in \Dc$ and $u(z,\gamma z) < 4 c_0 \delta^2$ for some $\gamma \in \tmop{SL}_2(\mathbb Z)\setminus\{\pm I\}$  where $c_0$ is sufficiently small then $z \notin \mathcal F_{\delta}$. For $\gamma =\begin{pmatrix}
a & b \\
c & d
\end{pmatrix} \in \tmop{SL}_2(\mathbb Z)$,
the inequalities
\begin{equation} \label{eq:dist1}
\frac{|y-y'|^2}{4yy'} \le u(z,\gamma z) < 4 c_0 \delta^2
\end{equation}
along with $y'=y/|cz+d|^2$ imply that for $y<2$
\begin{equation} \notag
%\label{eq:keyineq}
|cz+d|^2 |1-\tfrac{1}{|cz+d|^2}|^2 < 16 c_0 \delta^2.
\end{equation}
We infer that $|cz+d| \le 2$ since $c_0$ is sufficiently small. Also for $\gamma \neq \pm I$ and $z \in \Dc$ with $y <2$ which satisfies the second inequality in \eqref{eq:dist1}, we must have $c \neq 0$. We conclude that
\begin{equation} \label{eq:squeeze}
||z+\tfrac{d}{c}|^2-\tfrac{1}{c^2}|^2 <  \frac{ 64 c_0 \delta^2}{c^2}
\end{equation}

Since $y \ge \sqrt{3}/2$ we must have $|c|=1$ and consequently $d=-1,0$ or $1$ as $|x| \le 1/2$. For $$||z+1|^2-1| \ll \sqrt{c_0} \delta$$ and $z \in \Dc$ and we must have that $|z-e_3| \le \delta$. Similarly, if
$$||z-1|^2-1| \ll \sqrt{c_0} \delta$$ and $z \in \Dc$ we must have that $|z-e_1| \le \delta$.  If $c=\pm 1$, $d=0$ and  $u(z,\gamma z) < 4c_0 \delta^2$ then $\gamma z=\pm a-1/z$ and
$$u(z,\gamma z) \gg |x+\tfrac{x}{|z|^2} \mp a|^2,$$ so $a=-1,0$ or $1$ and $$|x+\tfrac{x}{|z|^2} \mp a| \ll \sqrt{c_0} \delta.$$
Also, by \eqref{eq:squeeze}, we have $||z|^2-1| \ll \sqrt{c_0} \delta$. We conclude that $|z-e_1| \le \delta$ or $|z-e_3|\le \delta$ if $a=\pm 1$ and $|z-e_2|\le \delta$ if $a=0$. This completes the proof of Theorem \ref{thm:bergman}.
\end{proof}

\subsection{Analysis of Bergman kernel near the cusp: Proof of Proposition \ref{prop:nearthecusp}}

In this section we analyze the Bergman kernel $R_k(z,z)$ near the cusp at infinity. In particular we will show that in the range $\sqrt{k}/(12\pi)< y <k/(2\pi)$ the size of the kernel $R_k(z,z)$ fluctuates and for $y>k/(2\pi)$ it is exponentially small.

\begin{theorem} \label{thm:nearthecusp}
Let $z=x+iy \in \Dc$ and $k \ge 4$ be even. Each of the following hold.
\begin{enumerate}[i.]
\item Uniformly, we have for any $A \ge 1$ that
\[
R_k(z,z) \ll_A \begin{cases}
1+\frac{y}{\sqrt{k}} & \text{ if }
\min_{n \in \mathbb Z} | n-\frac{k-1}{4\pi y}| \le \frac{\sqrt{k} \log k}{y}, \\
k^{-A} & \text{ otherwise.}
\end{cases}
\]
\item If $y \ge k/(2\pi)$ we have \[R_k(z,z) \ll e^{-k/37}.\]
\item For $k$ sufficiently large, we have
\[
R_k(z,z)\gg \displaystyle\begin{cases}
1 & \text{ if }  2 \le y \le \frac{\sqrt{k}}{12\pi}, \\
\frac{y}{\sqrt{k}} & \text{ if }  \frac{\sqrt{k}}{12\pi} < y < \frac{k}{2\pi}  \text{ and } \min_{n \in \mathbb N } \big|\frac{k-1}{4\pi y}-n\big| \le \frac{\sqrt{k-1}}{12 \pi y}.
\end{cases}
\]
\end{enumerate}
\end{theorem}

The upper bound given in Theorem \ref{thm:nearthecusp}(i.) refines previous upper bounds given for the Bergman kernel due to Steiner \cite[Proposition 3.2]{Steiner} and Aryasomayajula and Mukherjee \cite[Proposition 2.2]{am}.

Using Theorem \ref{thm:nearthecusp} we will quickly prove Proposition \ref{prop:nearthecusp}.

\begin{proof}[Proof of Proposition \ref{prop:nearthecusp} assuming Theorems \ref{thm:bergman}(ii.) and \ref{thm:nearthecusp}]
By \eqref{eq:relationrr} we have as $k \rightarrow \infty$
\[
r_k(z,z) =\bigg(\frac{3}{2 \pi}+o(1) \bigg)R_k(z,z).
\]
Applying Theorem \ref{thm:nearthecusp} and recalling \eqref{eq:rk covar def} with $z=w$ completes the proof of (i.)-(iii.). \end{proof}

Before proving Theorem \ref{thm:nearthecusp} we require several preliminary lemmas.

\begin{lemma} \label{lem:transform}
Let $k \ge 4$ be even.
For $z=x+iy \in \mathbb H^2$ with $y\ge 2$ we have
\[
R_k(z,z)
=\frac{2(4\pi y)^k}{\Gamma(k)}\sum_{m \ge 1} m^{k-1} e^{-4\pi m y}+O((1+y/3)^{-k+5}).
\]

\end{lemma}

\begin{proof}
We argue as in the proof of Theorem \ref{thm:bergman}. Let \[\Gamma_{\infty}= \bigg\{\begin{pmatrix}
\ast & \ast \\
0 &  \ast
\end{pmatrix} \in \tmop{SL}_2(\mathbb Z)
 \bigg\}.
\]
Using \eqref{eq:distance} we get that
\[
R_{k}(z,z)=\sum_{\gamma \in \Gamma_{\infty} } \ell_{\gamma}(z,z)^k
+O\bigg(\max_{\substack{\gamma \in \tmop{SL}_2(\mathbb Z) \\ \gamma \notin \Gamma_{\infty} }} (1+u(z,\gamma z))^{-k/2+2} \sum_{\gamma \in \tmop{SL}_2(\mathbb Z)}(1+u(z,\gamma z))^{-2} \bigg).
\]

For $\gamma =\begin{pmatrix}
a & b \\
c & d
\end{pmatrix}
\notin  \Gamma_{\infty} $ we have $c \neq 0$ and $y'=\Im \gamma z=y/|cz+d|^2 \le 1/y$. This implies for $y \ge 2$ and $\gamma \notin \Gamma_{\infty}$ that $u(z,\gamma z) \ge 9 y^2/64$. Hence, using \eqref{eq:sumbd} we conclude for $y \ge 2$ that
\begin{equation} \label{eq:cleanR}
R_{k}(z,z)=\sum_{\gamma \in \Gamma_{\infty} } \ell_{\gamma}(z,z)^k+O((1+9y^2/64)^{-k/2+2} \cdot y).
\end{equation}
The error term in \eqref{eq:cleanR} easily seen to be \begin{equation} \label{eq:cleanRerror}
\ll (1+y/3)^{-k+5}
\end{equation}
since $y \ge 2$.

For $\gamma^{\pm} =
\begin{pmatrix}
\pm 1 & n \\
 0 & \pm 1
\end{pmatrix}
$, we have that $\gamma^{\pm} z=z\pm n$ and
\[\ell_{ \gamma^{\pm}}( z,z)=\frac{2i y}{2iy \pm n}=\frac{1}{1\mp \frac{i n}{2y}}.\]
Hence using \eqref{eq:cleanR} and \eqref{eq:cleanRerror} we conclude, for $y \ge 2$
\begin{equation}  \label{eq:rclean2}
R_k(z,z)=2 \sum_{n \in \mathbb Z} \frac{1}{(1+\frac{i n}{2y})^k}+O\bigg( (1+y/3)^{-k+5} \bigg).
\end{equation}

 To evaluate the sum on the r.h.s. of \eqref{eq:rclean2} we will use Poisson summation. For $t\in \mathbb R$, let
\[
f(t)=t^{k-1} e^{-4\pi t} 1_{(0,\infty)}(t).
\]
For $\xi \in \mathbb R$ and $y>0$ we have $\Re(4\pi y+2\pi i\xi)>0$, so that applying \cite[Equation 3.478.1]{gradshteyn-ryzhik-2014}, with $p=1,\nu=k,\mu=4\pi y+2\pi i\xi$, yields
\[
\widehat f(\xi)=\int_0^{\infty} t^{k-1 } e^{-4\pi y t} e^{-2\pi i \xi t}\, dt=\frac{\Gamma(k)}{(4\pi y)^k(1+\frac{i\xi}{2y})^k}.
\]
In particular, for $k \ge 2$ we have $|f(t)|,|\widehat f(t)| \ll \frac{1}{1+|t|^2}$ where the implied constant depends on $k,z,w$.
Applying the Poisson summation formula we have
\begin{equation} \label{eq:poisson}
\begin{split}
 \frac{\Gamma(k)}{(4\pi y)^k}\sum_{n \in \mathbb Z} \frac{1}{(1+\frac{in}{2y})^{k}}=
\sum_{m \in \mathbb Z} m^{k-1} e^{-4\pi ny} 1_{(0,\infty)}(m).
% =&\sum_{n \in \mathbb Z} \int_0^{\infty} t^{k-1 } e^{-t(Y+2\pi i(n-\alpha))} \, dt\\
\end{split}
\end{equation}
Using \eqref{eq:poisson} in \eqref{eq:rclean2} completes the proof.
\end{proof}

\begin{lemma} \label{lem:incomplete}
Let $0< \Delta <1$ and $Y>0$. Then for real $\kappa \ge 1/2$
\begin{equation} \notag
\sum_{\substack{m \in \mathbb N \\|m -\frac{\kappa}{Y}|> \Delta \frac{\kappa}{Y}}} m^{\kappa} e^{-Ym}\ll \bigg(\frac{1}{ \Delta } + \kappa \bigg)\frac{1}{Y}\left( \frac{\kappa}{e Y}\right)^{\kappa} \exp\left( -\frac{\Delta^2}{4} \kappa \right) .
\end{equation}
\end{lemma}
\begin{proof}
Write $h_1(t)=\kappa \log t-t$ and $g_1(t)=h_1(t)-h_1(\kappa)$. Note that $\exp(h_1(\kappa))=\left( \frac{\kappa}{e}\right)^{\kappa}$. For $ t \ge (1+\Delta) \kappa$, noting $(1+\Delta)^{-1} \le 1-\Delta/2$ since $0< \Delta <1$, we have
\[
-1 < g_1'(t) \le  \frac{-\Delta}{2}.
\]
Also, the function $t^{\kappa}e^{-Yt}$ is strictly decreasing on $ (\kappa/Y, \infty)$. Write
\begin{equation*}
E(\kappa,Y)=\begin{cases}
\big((1+\Delta) \frac{\kappa}{Y} \big)^{\kappa} e^{-(1+\Delta)\kappa}  & \text{ if } Y \le (1+\Delta)\kappa,\\
e^{-Y}  & \text{ if } Y \ge (1+\Delta)\kappa
\end{cases}
\end{equation*}
and note the two terms are equal when $Y=(1+\Delta)\kappa$.
We have that
\[
\begin{split}
\sum_{\substack{m \in \mathbb N \\ m > (1+\Delta) \frac{\kappa}{Y}}} m^{\kappa} e^{-Ym} & \le \frac{1}{Y^{\kappa+1}} \int_{(1+\Delta)\kappa}^{\infty} t^{\kappa} e^{-t} \, dt+E(\kappa,Y) \\
& \le \left( \frac{\kappa}{Ye}\right)^{\kappa} \cdot \frac{2 }{Y\Delta} \int_{(1+\Delta)\kappa}^{\infty}(-g_1'(t)) \exp(g_1(t)) \, dt+E(\kappa,Y)\\
&=\left( \frac{\kappa}{Ye}\right)^{\kappa} \cdot \frac{2}{Y\Delta} \exp(g_1((1+\Delta)\kappa))+E(\kappa,Y).
\end{split}
\]

Noting $\log (1+t)-t \le -t^2/4$ for $0<t<1$,
it follows that
\begin{equation} \label{eq:g1bd}
g_1((1+\Delta)\kappa)=\kappa(\log(1+\Delta)-\Delta) \le -\Delta^2\kappa/4
\end{equation}
for $0 < \Delta <1$.
Hence,
\begin{equation} \label{eq:tailbd1}
\sum_{m >(1+\Delta) \frac{\kappa}{Y}} m^{\kappa} e^{-Ym} \le \frac{1}{Y\Delta}  \left( \frac{\kappa}{e Y}\right)^{\kappa} \exp\left( -\frac{\Delta^2}{4} \kappa \right)+E(\kappa,Y).
\end{equation}
For $Y< (1+\Delta)\kappa$, using \eqref{eq:g1bd} we have
\begin{equation} \label{eq:tailbd2}
E(\kappa,Y) = \bigg(\frac{\kappa}{eY} \bigg)^{\kappa} \exp(g_1((1+\Delta)\kappa)) \le   \left( \frac{\kappa}{e Y}\right)^{\kappa} \exp\left( -\frac{\Delta^2}{4} \kappa \right) \ll \frac{\kappa }{Y} \left( \frac{\kappa}{e Y}\right)^{\kappa} \exp\left( -\frac{\Delta^2}{4} \kappa \right),
\end{equation}
and this bound for $E(\kappa,Y)$ also holds in the range $Y\ge (1+\Delta)\kappa $ since it is not hard to see that the r.h.s. of \eqref{eq:tailbd2} is $\gg e^{-Y}$ in this range.

Next, we assume $(1-\Delta)\kappa/Y > 1$.
We use that $t^{\kappa}e^{-tY}
$ is strictly increasing on $(0,\kappa/Y)$ and $g_1'(t)\ge \Delta$ for $0<t<(1-\Delta)\kappa$ to see that
\[
\begin{split}
\sum_{1 \le m < (1-\Delta)\frac{\kappa}{Y}} m^{\kappa} e^{-mY}
\le& \frac{1}{Y^{\kappa +1}}\int_0^{(1-\Delta)\kappa} t^{\kappa} e^{-t} \, dt+ \bigg((1-\Delta) \frac{\kappa}{Y} \bigg)^{\kappa} e^{-(1-\Delta)\kappa}\\
\le& \bigg( \frac{\kappa}{Ye}\bigg)^{\kappa} \cdot \frac{1}{Y\Delta} \int_0^{(1-\Delta)\kappa} g_1'(t) \exp(g_1(t)) \, dt+\bigg((1-\Delta) \frac{\kappa}{Y} \bigg)^{\kappa} e^{-(1-\Delta)\kappa}\\
=& \bigg(\frac{1}{\Delta Y}+1 \bigg) \bigg( \frac{\kappa}{Ye}\bigg)^{\kappa} \cdot  \exp(g_1((1-\Delta)\kappa)),
\end{split}
\]
where in the last step we used that
\[
g_1((1-\Delta)\kappa)=\kappa(\log(1-\Delta)+\Delta).
\]
Noting that \[g_1((1-\Delta)\kappa)=\kappa(\log (1-\Delta)+\Delta) \le -\Delta^2 \kappa/2\]
and using $\kappa/Y>1$ we conclude that
\begin{equation} \label{eq:smallbd}
\sum_{1 \le m < (1-\Delta)\frac{\kappa}{Y}} m^{\kappa} e^{-mY} \ll \bigg(\frac{1}{ \Delta } +\kappa\bigg)\frac{1}{Y}\left( \frac{\kappa}{e Y}\right)^{\kappa} \exp\left( -\frac{\Delta^2}{2} \kappa \right).
\end{equation}
Combining \eqref{eq:tailbd1}, \eqref{eq:tailbd2}, and \eqref{eq:smallbd} we complete the proof.
\end{proof}

\begin{proof}[Proof of Theorem \ref{thm:nearthecusp}]
Write $\kappa=k-1$ and $Y=4\pi y$. Also, let $g_2(t)=h_2(t)-h_2(\kappa/Y)$ where $h_2(t)=\kappa \log t-Yt$. By
Lemmas \ref{lem:transform} and \ref{lem:incomplete} with $\Delta=1/3$ we have uniformly for $y \ge 2$
that
\begin{equation} \label{eq:Rcleanedup}
R_k(z,z)=\frac{2Y^{\kappa+1}}{\Gamma(\kappa+1)}\bigg(\frac{\kappa}{Ye}\bigg)^{\kappa}\sum_{\substack{|n-\frac{\kappa}{Y}| < \frac{\kappa}{3Y}}} \exp(g_2(n))+
O(e^{-k/37} ),
\end{equation}
where we have used Stirling's formula to bound the error term.
Notice that the sum is empty if $\kappa /Y<2/3$, this proves part (ii.).

It remains to prove parts (i.) and (iii.).
Using that $g_2(\tfrac{\kappa}{Y})=g_2'(\tfrac{\kappa}{Y})=0$ and $g_2''(t)=-\kappa/t^2$, we have for $|t-\tfrac{\kappa}{Y}| < \tfrac{\kappa}{3Y}$ that
\begin{equation} \label{eq:taylor}
\frac{-9Y^2}{8\kappa}\bigg( t-\frac{\kappa}{Y}\bigg)^2
\le g_2(t) \le \frac{-9Y^2}{32\kappa}\bigg( t-\frac{\kappa}{Y}\bigg)^2 .
\end{equation}
Equation \eqref{eq:taylor} implies 
\begin{equation} \label{eq:usefulbd}
\sum_{\substack{|n-\frac{\kappa}{Y}| < \frac{\kappa}{3Y}}} \exp(g_2(n))
\le \sum_{n \in \mathbb Z}
\exp\bigg( \frac{-9Y^2}{32\kappa}\bigg( n-\frac{\kappa}{Y}\bigg)^2\bigg)  \ll_A
\begin{cases}
1+\frac{\sqrt{\kappa}}{Y} & \text{ if }
\min_{n \in \mathbb Z} | n-\frac{\kappa}{Y}| \le \frac{\sqrt{\kappa} \log \kappa}{Y}, \\
k^{-A} & \text{ otherwise,}
\end{cases}
\end{equation}
for any $A\ge 1$.
Using \eqref{eq:usefulbd} in \eqref{eq:Rcleanedup} along with Stirling's formula we have for $2 \le y \le k/(2\pi)$ that
\[
R_k(z,z) \ll \frac{Y}{\sqrt{\kappa}}\sum_{\substack{|n-\frac{\kappa}{Y}| < \frac{\kappa}{3Y}}} \exp(g_2(n)))+e^{-k/37} \ll \begin{cases}
1+\frac{Y}{\sqrt{\kappa}} & \text{ if }
\min_{n \in \mathbb Z} | n-\frac{\kappa}{Y}| \le \frac{\sqrt{\kappa} \log \kappa}{Y}, \\
k^{-A} & \text{ otherwise,}
\end{cases}
\]
which proves (i.) if $ 2 \le y \le k/(2\pi)$. If $z \in \Dc$ and $y \le 2$ then (i.) follows from Theorem \ref{thm:bergman} and if $y \ge k/(2\pi)$ then (i.) follows from (ii.).

To prove (iii.) we use \eqref{eq:taylor} and Stirling's formula to see that
\begin{equation} \label{eq:finalbd}
\begin{split}
\frac{Y^{\kappa+1}}{\Gamma(\kappa+1)}\bigg(\frac{\kappa}{Ye}\bigg)^{\kappa}\sum_{\substack{|n-\frac{\kappa}{Y}| < \frac{\kappa}{3Y}}} \exp(g_2(n))
\gg& \frac{Y}{\sqrt{\kappa}}
\sum_{|n-\frac{\kappa
}{Y}| \le \frac{\kappa}{ 3Y}}
\exp\bigg( \frac{-9Y^2}{8\kappa}\bigg( n-\frac{\kappa}{Y}\bigg)^2\bigg)\\
\gg&
\frac{Y}{\sqrt{\kappa}}
\sum_{|n-\frac{\kappa
}{Y}| \le \frac{\sqrt{\kappa}}{ 3Y}}
1
\end{split}
\end{equation}
where we have used positivity of the summands to drop terms in the sum.
The sum on the r.h.s. of \eqref{eq:finalbd} is nonempty if and only if
\[ \min_{n \in \mathbb N } \bigg|\frac{k-1}{4\pi y}-n\bigg| \le \frac{\sqrt{k-1}}{12 \pi y}\]
and is $\gg \tfrac{\sqrt{\kappa}}{Y}$ if $2 \le y \le \tfrac{\sqrt{k}}{12\pi}$.
Using \eqref{eq:finalbd} in \eqref{eq:Rcleanedup} completes the proof of (iii.). \end{proof}

\section{\texorpdfstring{$L^p$}{Lp} norm estimates and concentration around the median: Proof of Theorem \ref{thm:mean mag exp conc}(ii.) and Theorem \ref{thm:sup bnd glob}(ii.)}
\label{sec:low bnd conc}

\subsection{General setting}

Given a smooth bounded function $f:\Dc\to\mathbb{C}$ and $p\ge1$,
denote the $L^{p}$ norm of $f$ restricted to a subdomain $\Dc'\subseteq\Dc$ by
\[
\left\Vert f\right\Vert _{p,\Dc'}:=\left(\int_{\Dc'}\left|f\left(z\right)\right|^{p}\,\frac{dxdy}{y^{2}}\right)^{1/p}.
\]
We also denote the sup norm on $\Dc'$ by
\begin{equation} \notag
%\label{eq:deflpnorm}
\left\Vert f\right\Vert _{\infty,\Dc'}=\sup_{z\in\Dc'}\left|f\left(z\right)\right|,
\end{equation}
and the random fields
\begin{equation}
\label{eq:Fk=y^k/2 gk}
h_{k}\left(z\right):=y^{k/2}g_{k}\left(z\right),
\end{equation}
as in \eqref{eq:fk spher field def}.
Under the newly introduced notation, we have $\M_{k} = \left\Vert h_{k}\right\Vert _{\infty,\Kc}$ and $\M^g_{k} = \left\Vert h_{k}\right\Vert _{\infty,\Dc}$.

In this section we prove the following result for the expected value of the $L^p$ norm of $h_k$, which is given in terms of the covariance function $r_k$ as defined in
\eqref{eq:rk covar def}. To state the result we define $r_{k,\tmop{diag}}:\mathbb C \rightarrow \mathbb R_{\ge 0}$ by
\begin{equation} \label{eq:rkdiagdef}
r_{k,\tmop{diag}}(z)=r_k(z,z).
\end{equation}

\begin{theorem} \label{thm:expectationlpnorm}
Each of the following hold.
\begin{enumerate}
\item[i.] Let $\Kc$ be a compact set with positive area. Then for any $p \ge 2$
\[
\mathbb{E}\left[\left\Vert h_{k}\right\Vert _{p,\Kc}\right]=
 \left\Vert \sqrt{r_{k,\tmop{diag}}} \right\Vert_{p,\mathcal \Kc} \sqrt{N} \bigg(\frac{\Gamma\left(\frac{p}{2}+1\right)\Gamma\left(\frac{2N-1}{2}\right)}{\Gamma\left(\frac{p+2N-1}{2}\right)} \bigg)^{1/p}+O_{\Kc}\bigg( \frac{\sqrt{p}}{k^{1/p}}\bigg).
\]
\item[ii.] For any $p \ge 2$ 
\[
\mathbb{E}\left[\left\Vert h_{k}\right\Vert _{p,\Dc}\right]=
 \left\Vert \sqrt{r_{k,\tmop{diag}}} \right\Vert_{p,\Dc} \sqrt{N} \bigg(\frac{\Gamma\left(\frac{p}{2}+1\right)\Gamma\left(\frac{2N-1}{2}\right)}{\Gamma\left(\frac{p+2N-1}{2}\right)} \bigg)^{1/p}+O\big( k^{\frac14-\frac{3}{2p}}\sqrt{p}\big).
\]
\end{enumerate}

\end{theorem}

% In particular, the lower bound of Theorem \ref{thm:mean mag exp conc}(i.) reads
% \begin{equation}
% \label{eq:low bnd not}
% \mathbb{E}\left[\left\Vert h_{k}\right\Vert _{\infty,\Kc}\right]\gg_{\Kc}\sqrt{\log k},
% \end{equation}
% where we recall that $\Kc\subseteq\Dc$ is a fixed compact subdomain of $\Dc$ of positive area.
% Throughout the proof of the lower bound \eqref{eq:low bnd not}, we may assume, without loss of generality, that $\Kc$ lies outside the excised neighborhood of the elliptic points $e_1, e_2, e_3$, i.e., that $\Kc \subseteq \mathcal{F}_\delta$. Indeed, if that is not the case, we can restrict to the domain $\Kc \cap \mathcal{F}_\delta \subseteq \Kc$ which is also of positive area for $\delta>0$ sufficiently small.

% Trivially, for any $p\ge2$, we have
% \[
% \left\Vert h_{k}\right\Vert _{p,\Kc}\ll\left\Vert h_{k}\right\Vert _{\infty,\Kc},
% \]
% and therefore
% \[
% \mathbb{E}\left[\left\Vert h_{k}\right\Vert _{\infty,\Kc}\right]\gg\mathbb{E}\left[\left\Vert h_{k}\right\Vert _{p,\Kc}\right].
% \]
% Thus, constructing a number $p>2$ so that \[\mathbb{E}\left[\left\Vert h_{k}\right\Vert _{p,\Kc}\right] \gg _{\Kc}\sqrt{\log k}\]
% would allow for inferring the desired lower bound \eqref{eq:low bnd not} of Theorem \ref{thm:mean mag exp conc}(i.).

\subsection{Some auxiliary results}

Throughout this section we let $\mathcal D'$ denote a subset of $\mathcal D$ of positive area. Also, let $\lk=L(\mathcal D',k)$ be a positive real number such that
\begin{equation}
\label{eq:RudnickBound}
\left\Vert f \right\Vert_{\infty,\mathcal D'}\le \lk \left\Vert f \right\Vert _{2}
\end{equation}
holds uniformly for all $f \in \Sc_k$. For $\mathcal D'=\Kc$
the sup norm bound \eqref{eq:RudnickBoundIntro} implies that we may choose $\lk $ so that
\begin{equation} \label{eq:Lbd1}
\lk \ll_{\Kc} k^{1/2}.
\end{equation}
Additionally, using \eqref{eq:SteinerBoundIntro} shows for $\mathcal D'=\mathcal D$ we may take $\lk$ so that
\begin{equation} \label{eq:Lbd2}
\lk \ll k^{3/4}.
\end{equation}

\begin{lemma}
\label{lem:LpBound}
Let $g_{k}^{1},g_{k}^{2}  \in \Sc_k$,
and denote $h_{k}^{1}\left(z\right)=y^{k/2}g_{k}^{1}\left(z\right)$,
$h_{k}^{2}\left(z\right)=y^{k/2}g_{k}^{2}\left(z\right)$. Then
for every
$p\ge2$, we have
\[
\left\Vert h_{k}^{1}-h_{k}^{2}\right\Vert _{p,\Dc'}\le  \lk^{1-\frac{2}{p}}\left\Vert h_{k}^{1}-h_{k}^{2}\right\Vert _{2}.
\]

\end{lemma}

\begin{proof}
We have
\[
\left\Vert f\right\Vert _{p,\Dc'}\le\left\Vert f\right\Vert _{\infty,\Dc'}^{1-\frac{2}{p}}\left\Vert f^{\frac{2}{p}}\right\Vert _{p,\Dc'}=\left\Vert f\right\Vert _{\infty,\Dc'}^{1-\frac{2}{p}}\cdot \left\Vert f\right\Vert _{2,\Dc'}^{\frac{2}{p}}.
\]
An application of this inequality on $h_{k}^{1}-h_{k}^{2}$ gives
\begin{align}
\left\Vert h_{k}^{1}-h_{k}^{2}\right\Vert _{p,\Dc'} & \le\left\Vert h_{k}^{1}-h_{k}^{2}\right\Vert _{\infty,\Dc'}^{1-\frac{2}{p}}\cdot \left\Vert h_{k}^{1}-h_{k}^{2}\right\Vert _{2,\Dc'}^{\frac{2}{p}}\le
\left\Vert h_{k}^{1}-h_{k}^{2}\right\Vert _{\infty,\Dc'}^{1-\frac{2}{p}}\cdot \left\Vert h_{k}^{1}-h_{k}^{2}\right\Vert _{2}^{\frac{2}{p}}.\label{eq:InterpG1G2}
\end{align}
Substituting the sup norm bound \eqref{eq:RudnickBound} into \eqref{eq:InterpG1G2} yields the desired bound.
\end{proof}
We now prove a concentration inequality for $\left\Vert h_{k}\right\Vert _{p,\Dc'}$
around its median value, which we denote by $\mu_{k,p}=\mu\left(\left\Vert h_{k}\right\Vert _{p,\Dc'}\right)$.
\begin{lemma}
\label{lem:Concen}
Let $p\ge2$.
For any $r>0$, we have
\begin{equation} \label{eq:LpConcen}
\prob\left(\left|\left\Vert h_{k}\right\Vert _{p,\Dc'}-\mu_{k,p}\right|>r\right)\le 2e^{-c k L^{\frac{4}{p}-2} r^2}
\end{equation}
for some absolute constant $c>0$.
\end{lemma}

\begin{proof}
Recall that by L\'evy's inequality (see, e.g., \cite[Proposition 1.3 and Theorem 2.3]{Ledoux}),
if $X:\Scc_{\mathbb{C}}^{2N-1}\to\mathbb{R}$ is a Lipschitz continuous function (with respect to the geodesic
distance on the sphere) with a Lipschitz constant $C>0$, then for any
$r>0$ we have
\begin{equation}
\prob\left(\left|X-\mu\left(X\right)\right|>r\right)\le2e^{-\left(2N-2\right)\frac{r^{2}}{2C^{2}}}.\label{eq:prob_conc}
\end{equation}
In what follows we identify a function $h_{k}$ with the point $\left(a_{j}\right)_{j\le N}\in \Scc_{\mathbb{C}}^{2N-1}$
determined by its coefficients via \eqref{eq:Fk=y^k/2 gk} and \eqref{eq:gk rand cusp form}.

For
$p\ge2,$ define $X_{p}\left(h_{k}\right)=\left\Vert h_{k}\right\Vert _{p,\Dc'}$.
Then, by Lemma \ref{lem:LpBound}, we have
\begin{align*}
\left|X_{p}\left(h_{k}^{1}\right)-X_{p}\left(h_{k}^{2}\right)\right| & \le\left\Vert h_{k}^{1}-h_{k}^{2}\right\Vert _{p,\Dc'}\le \lk^{1-\frac2p}\left\Vert h_{k}^{1}-h_{k}^{2}\right\Vert _{2}\\
 & \le \lk^{1-\frac2p} \text{dist}\left(h_{k}^{1},h_{k}^{2}\right),
\end{align*}
where we denoted by $\text{dist}\left(\cdot,\cdot\right)$
the geodesic distance on $\Scc_{\mathbb{C}}^{2N-1}$. Thus,
one has for every $r>0$
\[
\prob\left(\left|\left\Vert h_{k}\right\Vert _{p,\Dc'}-\mu_{k,p}\right|>r\right)\le
2\exp\bigg(-\left(2N-2\right)\frac{r^{2}}{2 \lk^{2-\frac4p}}\bigg),
\]
which proves \eqref{eq:LpConcen} on recalling that $N= \frac{k}{12} + O(1)$. \end{proof}
\begin{corollary}
\label{cor:Exp_Med_diff}Let $p\ge2$.
We have
\begin{equation}
\label{eq:diff_compact}
\mathbb{E}\left[\left\Vert h_{k}\right\Vert _{p,\Dc'}\right]=\mu_{k,p}+O\left( \frac{L^{1-\frac2p}}{\sqrt{k}}\right).
\end{equation}

\end{corollary}

\begin{proof}
By \eqref{eq:LpConcen}, we have
\begin{align*}
\left|\mathbb{E}\left[\left\Vert h_{k}\right\Vert _{p,\Dc'}\right]-\mu_{k,p}\right| & \le\int_{\Scc_{\mathbb{C}}^{2N-1}}\left|\left\Vert h_{k}\right\Vert _{p,\Dc'}-\mu_{k,p}\right|\,d\left(\prob\right)=\int_{0}^{\infty}\prob\left(\left|\left\Vert h_{k}\right\Vert _{p,\Dc'}-\mu_{k,p}\right|>r\right)\,dr\\
 & \le2\int_{0}^{\infty}e^{-c k L^{\frac4p-2}r^{2}}\,dr=\sqrt{\frac{\pi}{c k}} L^{1-\frac2p}
\end{align*}
so that \eqref{eq:diff_compact} holds.

\end{proof}

% Corollary \ref{cor:Exp_Med_diff}shows that the lower bound \eqref{eq:low bnd not} of Theorem \ref{thm:mean mag exp conc}(i.) will follow, provided that we find a number $p=p(k)>2$ so that $\mu_{k,p}\gg_{\Kc}\sqrt{\log k}.$

\subsection{Concentration around the median: Proof of Theorem \ref{thm:mean mag exp conc}(ii.) and Theorem \ref{thm:sup bnd glob}(ii.)}

We turn to proving Theorem \ref{thm:mean mag exp conc}(ii.) and Theorem \ref{thm:sup bnd glob}(ii.), analogous to Lemma \ref{lem:Concen} and Corollary \ref{cor:Exp_Med_diff}
with $p=\infty$.
\begin{proof}[Proof of Theorem \ref{thm:mean mag exp conc}(ii.) and Theorem \ref{thm:sup bnd glob}(ii.).]
Recall that a function $h_{k}$ is identified with the point $(a_{j})\in \Scc_{\mathbb{C}}^{2N-1}$
determined by its coefficients via \eqref{eq:Fk=y^k/2 gk} and \eqref{eq:gk rand cusp form}. Define $X_{\infty}\left(h_{k}\right)=\left\Vert h_{k}\right\Vert _{\infty,\Kc}$.
Then by the sup norm bound \eqref{eq:Lbd1}, we have,
\begin{align*}
\left|X_{\infty}\left(h_{k}^{1}\right)-X_{\infty}\left(h_{k}^{2}\right)\right| & \le\left\Vert h_{k}^{1}-h_{k}^{2}\right\Vert _{\infty,\Kc}\le Ck^{\frac{1}{2}}\left\Vert h_{k}^{1}-h_{k}^{2}\right\Vert _{2}\\
 & \le Ck^{\frac{1}{2}}\cdot \text{dist}\left(h_{k}^{1},h_{k}^{2}\right)
\end{align*}
with some constant $C=C\left(\Kc\right)>0$. Thus, by \eqref{eq:prob_conc},
for any $r>0$, we have
\begin{equation}
\prob\left(\left|\left\Vert h_{k}\right\Vert _{\infty,\Kc}-\mu_{k}\right|>r\right)\le
2e^{-\left(2N-2\right)\frac{r^{2}}{2C^{2}\cdot k}}\le2e^{-cr^{2}}\label{eq:LpConcen-sup}
\end{equation}
for some constant $c=c\left(\Kc\right)>0$, which yields the inequality \eqref{eq:exp conc} of Theorem \ref{thm:mean mag exp conc}(ii.).

Now, by \eqref{eq:LpConcen-sup}, we have
\begin{align*}
\left|\mathbb{E}\left[\left\Vert h_{k}\right\Vert _{\infty,\Kc}\right]-\mu_{k}\right| & \le\int_{\Scc_{\mathbb{C}}^{2N-1}}\left|\left\Vert h_{k}\right\Vert _{\infty,\Kc}-\mu_{k}\right|\,d\left(\prob\right)=\int_{0}^{\infty}\prob\left(\left|\left\Vert h_{k}\right\Vert _{\infty,\Kc}-\mu_{k}\right|>r\right)\,dr\\
 & \le2\int_{0}^{\infty}e^{-c_{}r^{2}}\,dr=\sqrt{\frac{\pi}{c}}.
\end{align*}
The estimate \eqref{eq:mean_median_diff} of Theorem \ref{thm:mean mag exp conc}(ii.) follows.

For the global case, define $X_{\infty}^g\left(h_{k}\right)=\left\Vert h_{k}\right\Vert _{\infty,\Dc}$. Then by \eqref{eq:Lbd2} we have
\begin{align*}
\left|X^g_{\infty}\left(h_{k}^{1}\right)-X^g_{\infty}\left(h_{k}^{2}\right)\right| & \le\left\Vert h_{k}^{1}-h_{k}^{2}\right\Vert _{\infty,\Dc}\le Ck^{\frac{3}{4}}\left\Vert h_{k}^{1}-h_{k}^{2}\right\Vert _{2}\\
 & \le Ck^{\frac{3}{4}}\cdot \text{dist}\left(h_{k}^{1},h_{k}^{2}\right)
\end{align*}
with some constant $C>0$. Thus, by \eqref{eq:prob_conc},
for any $r>0$, we have
\begin{equation} 
\prob\left(\left|\left\Vert h_{k}\right\Vert _{\infty,\Dc}-\mu^g_{k}\right|>r\right)\le
2e^{-\left(2N-2\right)\frac{r^{2}}{2C^{2}\cdot k^{3/2}}}\le2e^{-\frac{cr^{2}}{k^{1/2}}}\label{eq:LpConcen-sup_glob}
\end{equation}
for some absolute constant $c>0$, which gives the inequality \eqref{eq:exp conc glob} of Theorem \ref{thm:sup bnd glob}(ii.).

By \eqref{eq:LpConcen-sup_glob}, we have
\begin{align*}
\left|\mathbb{E}\left[\left\Vert h_{k}\right\Vert _{\infty,\Dc}\right]-\mu^g_{k}\right| & \le\int_{\Scc_{\mathbb{C}}^{2N-1}}\left|\left\Vert h_{k}\right\Vert _{\infty,\Dc}-\mu^g_{k}\right|\,d\left(\prob\right)=\int_{0}^{\infty}\prob\left(\left|\left\Vert h_{k}\right\Vert _{\infty,\Dc}-\mu^g_{k}\right|>r\right)\,dr\\
 & \le2\int_{0}^{\infty}e^{-\frac{cr^{2}}{k^{1/2}}}\,dr=\sqrt{\frac{\pi}{c}}k^{1/4},
\end{align*}
which gives the estimate \eqref{eq:mean_median_diff_glob} of Theorem  \ref{thm:sup bnd glob}(ii.).
\end{proof}

\subsection{The expected value of the \texorpdfstring{$L^p$}{Lp} norm and the proof of Theorem \ref{thm:expectationlpnorm}}
\label{sec:lower bnd exp proc}

We first establish a precise formula for the expected value of 
$\left\Vert h_{k} \right\Vert_{p, \mathcal D'}^p$.

% We now proceed with the proof of the lower bound of Theorem \ref{thm:mean mag exp conc}(i.). Recall that we have
% reduced proving \eqref{eq:low bnd not} to the problem of finding a moment $p>2$ whose median satisfies $$\mu_{k,p}\gg_{\Kc}\sqrt{\log k}.$$
% The key arithmetic input is the uniform asymptotics \eqref{eq:var asuymp Dk sum} for the variance of $h_{k}(\cdot)$ of Proposition \ref{prop:Fk covar},
% valid outside the excised neighborhood of the elliptic points.
%
\begin{lemma}
\label{lem:LpAsymp}Let $p\ge1$ and $r_{k,\tmop{diag}}$ be as in \eqref{eq:rkdiagdef}. We have
\[
\left(\mathbb{E}\left[\left\Vert h_{k}\right\Vert _{p,\Dc'}^{p}\right]\right)^{1/p}=
 \left\Vert \sqrt{r_{k,\tmop{diag}}} \right\Vert_{p,\mathcal D'} \sqrt{N} \bigg(\frac{\Gamma\left(\frac{p}{2}+1\right)\Gamma\left(\frac{2N-1}{2}\right)}{\Gamma\left(\frac{p+2N-1}{2}\right)} \bigg)^{1/p}.
\]

\end{lemma}

\begin{proof}
A simple calculation shows (see, e.g., \cite[Appendix A]{BurqLebeau}),
that for every $0\le\theta\le\frac{\pi}{2}$, we have
\[
\prob\left(\left\{(a_1,\dots,a_N)\in \Scc_{\mathbb{C}}^{2N-1} :\: \left|a_{1}\right|>\cos\theta\right\}\right)=\left(\sin\theta\right)^{2N-3}.
\]
Denote
\[
v_{k}\left(z\right)=\left(y^{k/2}f_{j}\left(z\right)\right)_{j=1}^{N}
\]
which is a vector of length
\begin{equation} \label{eq:length}
\left|v_{k}\left(z\right)\right|=y^{k/2}\sqrt{\sum_{j=1}^{N}\left|f_{j}\left(z\right)\right|^{2}}=\sqrt{N} \sqrt{r_{k,\tmop{diag}}(z)}
\end{equation}
by \eqref{eq:rkdef} and is nonzero in $\Dc$ except for at most a finite number of points.
Letting $u=\left(a_{j}\right)_{j=1}^{N}\in \Scc_{\mathbb{C}}^{2N-1}$, one has
\[
h_{k}\left(z\right)=y^{k/2}\sum_{j=1}^{N}a_{j}f_{j}\left(z\right)=\left|v_{k}\left(z\right)\right|\cdot\left\langle u,\frac{\overline{v_{k}\left(z\right)}}{\left|v_{k}\left(z\right)\right|}\right\rangle .
\]

Then, by the invariance of the uniform measure on $\Scc_{\mathbb{C}}^{2N-1}$ with respect to rotations,
one may assume that the vector $\frac{\overline{v_{k}\left(z\right)}}{\left|v_{k}\left(z\right)\right|}$
is lying on the $z_{1}$ axis if $v_k(z) \neq 0$. Hence, for every angle $0\le\theta\le\frac{\pi}{2}$
one has
\[
\prob\left(\left|h_{k}\left(z\right)\right|>\left|v_{k}\left(z\right)\right|\cos\theta\right)=\prob\left(\left|z_{1}\right|>\cos\theta\right)=\left(\sin\theta\right)^{2N-3},
\]
if $v_k(z) \neq 0$.
Since
\[
\mathbb{E}\left[\left|h_{k}\left(z\right)\right|^{p}\right]=p\int_{0}^{\infty}\lambda^{p-1}\prob\left(\left|h_{k}\left(z\right)\right|>\lambda\right)\,d\lambda,
\]
and using \eqref{eq:length}
we may deduce that
\begin{align*}
\mathbb{E}\left[\left\Vert h_{k}\right\Vert _{p,\Dc'}^{p}\right] & =\int_{\Scc_{\mathbb{C}}^{2N-1}}\int_{\Dc'}\left|h_{k}\left(z\right)\right|^{p}\,\frac{dxdy}{y^{2}}\,d\left(\prob\right)=\int_{\Dc'}\mathbb{E}\left[\left|h_{k}\left(z\right)\right|^{p}\right]\,\frac{dxdy}{y^{2}}\\
 & =p\int_{\Dc'}\int_{0}^{\infty}\lambda^{p-1}\prob\left(\left|h_{k}\left(z\right)\right|>\lambda\right)\,d\lambda\,\frac{dxdy}{y^{2}}\\
 & =p\int_{\Dc'}\int_{0}^{\pi/2}\left(\left|v_{k}\left(z\right)\right|\cos\theta\right)^{p-1}\prob\left(\left|h_{k}\left(z\right)\right|>\left|v_{k}\left(z\right)\right|\cos\theta\right)\left|v_{k}\left(z\right)\right|\sin\theta\,d\theta\,\frac{dxdy}{y^{2}}\\
 & =p N^{p/2} \int_{0}^{\pi/2}\left(\cos\theta\right)^{p-1}\left(\sin\theta\right)^{2N-2}\,d\theta\, \int_{\mathcal D'} r_{k,\tmop{diag}}(z)^{p/2} \frac{dxdy}{y^{2}}\\
     & =\frac{p}{2} N^{p/2} \left\Vert \sqrt{r_{k,\tmop{diag}}} \right\Vert_{p,\mathcal D'}^p B\left(\frac{p}{2},\frac{2N-1}{2}\right) =\frac{p}{2} N^{p/2} \left\Vert \sqrt{r_{k,\tmop{diag}}} \right\Vert_{p,\mathcal D'}^p \frac{\Gamma\left(\frac{p}{2}\right)\Gamma\left(\frac{2N-1}{2}\right)}{\Gamma\left(\frac{p+2N-1}{2}\right)},
\end{align*}
as claimed.
% By a simple calculation using Stirling's approximation, we conclude
% that as $p\to\infty$ with $p=o\left(k\right)$, we have
% \begin{align*}
% \left(\mathbb{E}\left[\left\Vert h_{k}\right\Vert _{p,\Dc'}^{p}\right]\right)^{1/p} & =\left(\frac{p}{2}C_{\Dc'}\right)^{1/p}\cdot\left(\frac{k}{4\pi}\right)^{1/2}\left(\Gamma\left(\frac{p}{2}\right)\right)^{1/p}\left(\frac{2}{2N-1}\right)^{1/2}\left(1+o_{\Dc'}\left(1\right)\right)\\
%  & =\left(\frac{k}{4\pi}\cdot\frac{p}{2e}\cdot\frac{2}{2N-1}\right)^{1/2}\left(1+o_{\Dc'}\left(1\right)\right)\\
%  & =\sqrt{\frac{3}{2\pi e}}\sqrt{p}\left(1+o_{\Dc'}\left(1\right)\right).
% \end{align*}
\end{proof}
We now bound the distance between $\left(\mathbb{E}\left[\left\Vert h_{k}\right\Vert _{p,\Dc'}^{p}\right]\right)^{1/p}$
and the median $\mu_{k,p}=\mu\left(\left\Vert h_{k}\right\Vert _{p,\Dc'}\right)$.

\begin{lemma}
\label{lem:Distance} Let $p \ge 2$. We have
\[
\left(\mathbb{E}\left[\left\Vert h_{k}\right\Vert _{p,\Dc'}^{p}\right]\right)^{1/p}=\mu_{k,p}+O\bigg( \frac{L^{1-\frac{2}{p}}\sqrt{p}}{\sqrt{k}}\bigg).
\]
\end{lemma}

\begin{proof}
As before, we will think of $h_{k}$ as a random point on $\Scc_{\mathbb{C}}^{2N-1}$,
and denote $X_{p}\left(h_{k}\right)=\left\Vert h_{k}\right\Vert _{p,\Dc'}$.
Then
\begin{align*}
\left|\left(\mathbb{E}\left[\left\Vert h_{k}\right\Vert _{p,\Dc'}^{p}\right]\right)^{1/p}-\mu_{k,p}\right|^{p} & =\left|\left\Vert X_{p}\right\Vert _{L^{p}\left(\Scc_{\mathbb{C}}^{2N-1}\right)}-\mu_{k,p}\right|^{p}\le\left\Vert X_{p}-\mu_{k,p}\right\Vert _{L^{p}\left(\Scc_{\mathbb{C}}^{2N-1}\right)}^{p}\\
 & =\mathbb{E}\left[\left|\left\Vert h_{k}\right\Vert _{p,\Dc'}-\mu_{k,p}\right|^{p}\right]\\
 & =p\int_{0}^{\infty}\lambda^{p-1}\prob\left(\left|\left\Vert h_{k}\right\Vert _{p,\Dc'}-\mu_{k,p}\right|>\lambda\right)\,d\lambda.
\end{align*}
Hence, by (\ref{eq:LpConcen}) we obtain
\[
\left|\left(\mathbb{E}\left[\left\Vert h_{k}\right\Vert _{p,\Dc'}^{p}\right]\right)^{1/p}-\mu_{k,p}\right|^{p}\le2p\int_{0}^{\infty}\lambda^{p-1} e^{-c k L^{\frac{4}{p}-2} \lambda^2}\,d\lambda=\frac{p\Gamma\left(\frac{p}{2}\right)}{(ck L^{\frac4p-2})^{p/2}}.
\]
Thus,
\[
\left|\left(\mathbb{E}\left[\left\Vert h_{k}\right\Vert _{p,\Dc'}^{p}\right]\right)^{1/p}-\mu_{k,p}\right|\le\frac{p^{1/p}\left(\Gamma\left(\frac{p}{2}\right)\right)^{\frac{1}{p}}}{\sqrt{ck L^{\frac4p-2}}}\ll \frac{L^{1-\frac{2}{p}}\sqrt{p}}{\sqrt{k}}.
\]
\end{proof}

% Combining Corollary \ref{cor:Exp_Med_diff} with Lemmas \ref{lem:LpAsymp}, \ref{lem:Distance} 
We will now deduce Theorem \ref{thm:expectationlpnorm}.

\begin{proof}[Proof of Theorem \ref{thm:expectationlpnorm}]
Combining Corollary \ref{cor:Exp_Med_diff} with Lemmas \ref{lem:LpAsymp}, \ref{lem:Distance} and applying \eqref{eq:Lbd1} establishes (i.). Similarly, by Corollary \ref{cor:Exp_Med_diff}, Lemmas \ref{lem:LpAsymp}, \ref{lem:Distance} and \eqref{eq:Lbd2} we obtain (ii.).
\end{proof}

\section{Proof of Theorem \ref{thm:mean mag exp conc}\texorpdfstring{(i.)}{(i.)} and Theorem \ref{thm:sup bnd glob}\texorpdfstring{(ii.)}{(ii.)}} \label{sec:proofs mean}

\subsection{Proof of the lower bounds in Theorem \ref{thm:mean mag exp conc}\texorpdfstring{(i.)}{(i.)} and Theorem \ref{thm:sup bnd glob}\texorpdfstring{(ii.)}{(ii.)}} Combining Theorem \ref{thm:expectationlpnorm}(i.) with the estimates for the covariance function $r_k$ given in Proposition \ref{prop:Fk covar} yields asymptotics for the $L^p$ norm of $h_k$. We will now use this result to obtain a lower bound on the expected value of the sup norm of $h_k$, which will establish the lower bound in Theorem \ref{thm:mean mag exp conc}(i.).

\begin{proof}[Proofs of Theorem \ref{thm:mean mag exp conc}(i.), lower bound]
Write $\Kc'=\Kc \cap \mathcal F_{\delta}$ where $\mathcal F_{\delta}$ is defined as in \eqref{eq:Fdeltadef} and assume $\delta>0$ is sufficiently small, depending on $\Kc$, so that $\Kc'$ has positive area. Recall $r_{k,\tmop{diag}}(z)=r_k(z,z)$ and $N=\frac{k}{12}+O(1)$. Thus, using \eqref{eq:var asuymp Dk sum} we get that
\[
 \left\Vert \sqrt{r_{k,\tmop{diag}}} \right\Vert_{p,\mathcal \Kc} \ge  \left\Vert \sqrt{r_{k,\tmop{diag}}} \right\Vert_{p,\mathcal \Kc'} = \tmop{area}_{\mathbb H^2}(\Kc')^{1/p} \sqrt{\frac{3}{\pi}}(1+o(1)).
\]
Hence, using Stirling's formula it follows for $ p \rightarrow \infty$ as $k \rightarrow \infty$ with $p =o(k)$ that
\begin{equation} \label{eq:lowerbdE}
\begin{split}
\left\Vert \sqrt{r_{k,\tmop{diag}}} \right\Vert_{p,\mathcal \Kc} \sqrt{N}\bigg(\frac{\Gamma\left(\frac{p}{2}+1\right)\Gamma\left(\frac{2N-1}{2}\right)}{\Gamma\left(\frac{p+2N-1}{2}\right)} \bigg)^{1/p} 
\ge & \tmop{area}_{\mathbb H^2}(\Kc')^{1/p} \sqrt{\frac{3}{\pi}}  \Gamma(\tfrac{p}{2}+1)^{1/p}(1+o(1)) \\
=&\tmop{area}_{\mathbb H^2}(\Kc')^{1/p} \sqrt{\frac{3p}{2\pi e}}(1+o(1)) \\
=& \sqrt{\frac{3p}{2\pi e}} (1+o_{\Kc}(1)) .
\end{split}
\end{equation}

By Theorem \ref{thm:expectationlpnorm}(i.) and \eqref{eq:lowerbdE} there exists $C=C(\Kc)>0$ such that
\begin{equation} \label{eq:lowerbdmu}
\begin{split}
\mathbb E\left[ \left\Vert h_k \right\Vert_{p,\Kc} \right]  \ge &\left\Vert \sqrt{r_{k,\tmop{diag}}} \right\Vert_{p,\mathcal \Kc} \sqrt{N}\bigg(\frac{\Gamma\left(\frac{p}{2}+1\right)\Gamma\left(\frac{2N-1}{2}\right)}{\Gamma\left(\frac{p+2N-1}{2}\right)} \bigg)^{1/p}- C \frac{\sqrt{p}}{k^{1/p}} \\
\ge & \sqrt{p}\bigg(\sqrt{\frac{3}{2\pi e}}(1+o_{\Kc}(1))-\frac{C}{k^{1/p}} \bigg)
\end{split}
\end{equation}
for any $p  \rightarrow \infty$ as $k \rightarrow \infty$ with $p=o(k)$.
Taking $p=\eta \log k$ where $\eta=\eta(\Kc)>0$ is fixed and sufficiently small in terms of $C$ we have by \eqref{eq:lowerbdmu} that
\begin{equation} \label{eq:mulowerbd}
\mathbb E\left[ \left\Vert h_k \right\Vert_{\eta \log k,\Kc} \right] \gg_{\Kc} \sqrt{\log k}.
\end{equation}
 To complete the proof, observe that
\[
 \left\Vert h_{k}\right\Vert _{p,\Kc} \ll \sup_{z \in \Kc} |h_k(z)|,
\]
which holds for any $p\ge 2$, hence by \eqref{eq:mulowerbd} we conclude
\[
\mathbb E\left[ \sup_{z \in \Kc} |h_k(z)| \right] \gg  \mathbb E \left[ \left\Vert h_{k}\right\Vert _{\eta \log k,\Kc} \right]\gg_{\Kc} \sqrt{\log k}.
\]\end{proof}

We now turn to the proof of the lower bound in Theorem \ref{thm:sup bnd glob}(ii.).
\begin{proof}[Proof of Theorem \ref{thm:sup bnd glob}(i.), lower bound]

Recall that $h_{k}(z) = y^{k/2}\cdot g_{k}(z)$, whose variance at $z\in\Dc$ is given by
\begin{equation*}
\E\left[ \left| h_{k}(z)  \right|^{2}  \right] = r_{k}(z,z).
\end{equation*}
Then, by invoking Proposition \ref{prop:nearthecusp}(iii.), it is possible to find a (deterministic) point $z_{k}=x_{k}+iy_{k}\in \Hb^{2}$ so that the variance of $h_{k}(z_{k})$ satisfies
\begin{equation}
\label{eq:var >> k^1/2}
\var(h_{k}(z_{k})) = \E\left[ \left| h_{k}(z_{k})  \right|^{2}  \right]  \gg k^{1/2}
\end{equation}
where the constant involved in the $``\gg"$-notation is absolute.
Indeed, it is easy to find a number $y_{k}\in \left( \frac{k}{8\pi}  ,\frac{k}{2\pi}\right)$, say, so that $\frac{k-1}{4\pi y_{k}}\in \Z$. Setting $z_{k}=x_{k}+iy_{k}\in\Dc$ with $x_{k}\in (0,1)$ arbitrary
and applying Proposition \ref{prop:nearthecusp}(iii.) will imply the estimate \eqref{eq:var >> k^1/2}.

Had $h_{k}(z_{k})$ been Gaussian, \eqref{eq:var >> k^1/2} would have implied the lower bound in \eqref{eq:glob exp sup <<>>k^1/4}, as a (real or complex) Gaussian random variable $Z$ satisfies
\begin{equation}
\label{eq:var to exp}
\E\left[|Z| \right] \gg \var(Z)^{1/2},
\end{equation}
with the constant involved in the ``$\gg$"-notation absolute. Unfortunately, this is not the case, and, in general one cannot infer the
analogue of \eqref{eq:var to exp} without some restrictive assumptions on $Z$. Fortunately, $h_{k}(\cdot)$ is intimately related to its Gaussian counterpart, see the discussion in \S~\ref{sec:spher vs Gauss} above.
Namely, let the (complex) Gaussian random variable
\begin{equation}
\label{eq:Z Gauss def}
Z = Z_{k}:= \frac{1}{\sqrt{N}}\sum\limits_{j=1}^{N}b_{j}y_{k}^{k/2}f_{j}(z_{k}),
\end{equation}
where the $b_{j}$ are standard complex Gaussian i.i.d., and bear in mind that the point $z_{k}=x_{k}+iy_{k}\in \Dc$ was chosen so that \eqref{eq:var >> k^1/2} holds.
(If, instead of fixing $z_{k}$, one varies $z$ in $\Dc$, the result is a Gaussian random field with the same covariance kernel $r_{k}(\cdot,\cdot)$ as for $h_{k}(\cdot)$, cf. ~\eqref{eq:Gk Gauss def}. 
However, for
our purpose, we will only need to evaluate it at a single point of maximal variance, so there is no need to introduce a random Gaussian field, and we may settle for a single Gaussian random variable.)

Comparing \eqref{eq:Z Gauss def} to the definition \eqref{eq:fk spher field def} of $h_{k}(z)$ and \eqref{eq:gk rand cusp form}, we may observe that
\begin{equation}
\label{eq:Z= zeta hk(zk)}
Z  \stackrel{\Lc}{=}\zeta \cdot h_{k}(z_{k}),
\end{equation}
akin to \eqref{eq:G_{k}=zeta gk}, where $\zeta=\zeta_{N}>0$ (with $N \asymp k$) is a random variable, independent of $h_{k}(\cdot)$, highly concentrated at $1$, 
whose mean is given precisely by \eqref{eq:mean zeta}. 
Evidently, $$\var(Z) =\E[|Z|^{2}] = \E[|h_{k}(z_{k})|^{2}]=r_{k}(z_{k},z_{k}) \gg k^{1/2},  $$ by \eqref{eq:var >> k^1/2}. The upshot is that $Z$ is a Gaussian random variable, of variance $\gg k^{1/2}$, hence
one is eligible to apply \eqref{eq:var to exp} so that to infer
\begin{equation}
\label{eq:E|Z|>>k^1/4}
\E[|Z|] \gg k^{1/4}.
\end{equation}

Together with \eqref{eq:E|Z|>>k^1/4} and \eqref{eq:mean zeta}, 
and the independence of $\zeta$ and $h_{k}(\cdot)$, \eqref{eq:Z= zeta hk(zk)} implies that
\begin{equation*}
\E[|h_{k}(z_{k})|] = \E[|Z|] \cdot \left[ 1- O\left( \frac{1}{N}  \right)  \right]\gg k^{1/4}.
\end{equation*}
Finally, we write
\begin{equation*}
\E\left[\sup\limits_{z\in\Dc}\left| h_{k}(z)\right|\right] \ge \E\left[\left| h_{k}\left(z_{k}\right)\right|\right]\gg k^{1/4},
\end{equation*}
which is the lower bound of Theorem \ref{thm:sup bnd glob}(i.). \end{proof}

\subsection{Proof of the upper bounds in Theorem \ref{thm:mean mag exp conc}\texorpdfstring{(i.)}{(i.)} and Theorem \ref{thm:sup bnd glob}\texorpdfstring{(ii.)}{(ii.)}} As before, using Theorem \ref{thm:expectationlpnorm} and Proposition \ref{prop:Fk covar} we obtain upper bounds for the expected value of the $L^p$ norm of $h_k$, which we  will use to bound the expected value of the sup norm of $h_k$. To pass from the $L^p$ of $h_k$ to its sup norm we require the following auxiliary result:

\begin{lemma}
\label{prop:cusp form perturbation bound}
Recall that, for $k\ge 1$ even, $\Sc_{k}$ is the space of the weight-$k$ cusp forms, and denote $\|f \|_{2}=\sqrt{\langle f,f\rangle_{PS}}$ to be the norm on $\Sc_{k}$ corresponding to the Petersson inner product.
Then, uniformly for all $f\in \Sc_{k}$ of unit norm $\|f\|_{2}=1$, and $z=x+iy,w=u+iv\in \Hb^{2}$ one has the estimate
\begin{equation*}
y^{k/2}f(z) = v^{k/2}f(w) + O\left(k^{7/4}\cdot \frac{|z-w|}{\min(y,v)}\right).
\end{equation*}

\end{lemma}

\begin{proof}

Recall the bound \eqref{eq:SteinerBoundIntro}, which states
\begin{equation} \notag
\sup_{z\in \Hb^{2}} \frac{y^{k/2}|f(z)|}{\lVert f \rVert_2} \ll k^{3/4}.
\end{equation}
Hence, if $$|z-w| > \frac{\min(y,v)}{8k},$$ the result follows.
It remains to consider the case
\begin{equation}
\label{eq:asmpt|z-w|<= (y,v)/8k}
|z-w| \le \frac{\min(y,v)}{8k}.
\end{equation}

\vspace{2mm}

Since the disk centered at $z$ with radius $r=y/(4k)$ is contained in $\Hb^{2}$, $f$ is holomorphic in this disk. Also, this disk contains $w$ by \eqref{eq:asmpt|z-w|<= (y,v)/8k}. Hence,
\begin{equation} \label{eq:analytic}
f(w)=\sum_{j=0}^{\infty} \bigg(\frac{1}{2\pi i} \int_{|s-z|=r} f(s) \frac{ds}{(s-z)^{j+1}} \bigg)(w-z)^j.
\end{equation}
Write $s=x_s+iy_s$. For $|s-z|=r$, also write $y_s=y+r\sin \theta_s$ with $\theta_s \in [0,2\pi)$. We have
\[
y_s^{k/2}=y^{k/2}\left(1+\tfrac{r}{y} \sin \theta_s\right)^{k/2}.
\]
Since $r=y/(4k)$, it follows that
\begin{equation} \label{eq:asympry}
\left(1+\tfrac{r}{y} \sin \theta_s\right)^{k/2} \asymp 1,
\end{equation}
and consequently $y_s^{k/2} \asymp y^{k/2}$ for $|s-z|=r$. Applying \eqref{eq:asympry} along with \eqref{eq:SteinerBoundIntro} yields
\begin{equation} \label{eq:derivbd}
\begin{split}
\bigg|\frac{1}{2\pi i} \int_{|s-z|=r} f(s) \frac{ds}{(s-z)^{j+1}} \bigg| \ll & \frac{\lVert f \rVert_2}{y^{k/2} r^{j+1}} \int_{|z-s|=r} \frac{y_s^{k/2}|f(s)|}{\lVert f \rVert_2} \, |ds| \\
\ll & \frac{\lVert f \rVert_2}{ y^{k/2}r^j} k^{3/4}.
\end{split}
\end{equation}

Using \eqref{eq:derivbd} in \eqref{eq:analytic} and recalling the assumption \eqref{eq:asmpt|z-w|<= (y,v)/8k}, which implies $|z-w|\le r/2$, we get
\begin{equation} \notag
f(w)=f(z)+O\bigg( \lVert f \rVert_2 k^{3/4} \frac{ |z-w|}{y^{k/2}r}\bigg).
\end{equation}
Hence, noting that
\[
y^{k/2}=v^{k/2}\bigg(1+\frac{y-v}{v}\bigg)^{k/2}=v^{k/2}\bigg( 1+O\bigg( \frac{k|y-v|}{v}\bigg)\bigg)
\]
and recalling the bound \eqref{eq:SteinerBoundIntro}, we have \[\frac{y^{k/2}f(z)}{\lVert f \rVert_2}=\frac{v^{k/2}f(w)}{\lVert f \rVert_2}+O\bigg(k^{7/4}  \frac{|z-w|}{\min(y,v)} \bigg),\]
which completes the proof.
\end{proof}

\begin{proof}[Proof of Theorem \ref{thm:mean mag exp conc}(i.), upper bound]
Let $Y=Y(\Kc)>0$ be sufficiently large so that $\Kc \subseteq \Dc'_Y$ where $\Dc'_Y=\{ z \in \Dc : \Im z\le Y\}$. Since $\Kc \subseteq \Dc'_Y$ it suffices to bound the sup norm of $h_k$ over $\Dc'_Y$. Also, let $z_0=z_0(h_k)\in \Dc'_Y$ be such that
\[
|h_k(z_0)|=\sup_{z\in  \Dc'_Y} |h_k(z)|.
\]
For $r>0$ let
\[
\mathcal R_k=\mathcal R_k(z_0,r)=\Dc'_Y \cap \{ z \in \mathbb C :
|z-z_0|\le r\}.
\]
Hence, using Lemma \ref{prop:cusp form perturbation bound}, there exists an absolute constant $C>0$ such that for any $p \ge 1$ 
\[
\left\Vert h_k \right\Vert_{p,\mathcal R_k} \ge \tmop{area}_{\mathbb H^2}
(\mathcal R_k)^{1/p}\bigg( \sup_{z \in \Dc'_Y} |h_k(z)|-C k^{7/4}r\bigg).
\]
We have that $\tmop{area}_{\mathbb H^2}(\mathcal R_k) \gg_{\Kc} r^2$ uniformly w.r.t. $z_0$.
It follows that for any $r > 0, p \ge 1$ we have
\begin{equation} \label{eq:expsupbd1}
\mathbb E\left[\sup_{z\in  \Dc'_Y} |h_k(z)|\right] \ll_{\Kc} \frac{1}{r^{2/p}} \mathbb E \left[\left\Vert
h_k \right\Vert_{p,\Dc'_Y}\right]+k^{7/4}r.
\end{equation}

By Theorem \ref{thm:expectationlpnorm}(i.) along with Proposition \ref{prop:nearthecusp}(i.), and recalling $N=\frac{k}{12}+O(1)$, we have for any $p \ge 2$ with $p=o(k)$
\begin{equation} \label{eq:expLpbd}
\begin{split}
  \mathbb{E}\left[\left\Vert h_{k}\right\Vert _{p,\Dc'_Y}\right]& \ll_{\Kc} \left\Vert \sqrt{r_{k,\tmop{diag}}} \right\Vert_{p,\Dc'_Y} \sqrt{N} \bigg(\frac{\Gamma\left(\frac{p}{2}+1\right)\Gamma\left(\frac{2N-1}{2}\right)}{\Gamma\left(\frac{p+2N-1}{2}\right)} \bigg)^{1/p}+\frac{\sqrt{p}}{k^{1/p}} \\
   & \ll_{\Kc} \sqrt{p} \left\Vert \sqrt{r_{k,\tmop{diag}}} \right\Vert_{p,\Dc'_Y}  +\sqrt{p}\ll_{\Kc} \sqrt{p},
  \end{split}
\end{equation}
where to obtain the second inequality we used standard bounds for the gamma function (cf. \eqref{eq:lowerbdE}).
Combining \eqref{eq:expsupbd1} and \eqref{eq:expLpbd} yields
\[
\mathbb E\left[\sup_{z\in  \Dc'_Y} |h_k(z)|\right]  \ll_{\Kc} \frac{\sqrt{p}}{r^{2/p}}+k^{7/4}r.
\]
Taking $r=k^{-7/4}$ and $p=\log k$ completes the proof.% \[
% \mathbb E[\mathcal L_k]
% \ll \frac{ k^{1/p}\sqrt{p}}{r^{2/p}}+k^{7/4}r.
% \]
% Taking $r=k^{-7/4}$ and $p=\log k$ we conclude that
% \[
% \mathbb E[\mathcal L_k] \ll \sqrt{\log k}.
% \]
\end{proof}

\begin{proof}[Proof of Theorem \ref{thm:sup bnd glob}(i.), upper bound]
First of all, it suffices to restrict to the domain \[\Dc_k'=\{z \in \Dc: \Im z \le k\}\] since for $z \in \Dc \setminus \Dc_k'$ it follows from Proposition \ref{prop:nearthecusp}(ii.) that
\begin{equation}\label{eq:cuspbd}
|h_k(z)| \le \bigg( \sum_{j=1}^N |a_j|^2 \bigg)^{1/2}  r_k(z,z)^{1/2} \ll e^{-k/75}.
\end{equation}
Arguing as in the proof of Theorem \ref{thm:mean mag exp conc}(i.), for $r>0$ we let
\[
\mathcal R_k=\mathcal R_k(z_0,r)=\Dc_k' \cap \{ z \in \mathbb C :
|z-z_0|\le r\},
\]
where $z_0 \in \mathcal D_k'$ satisfies $|h_k(z_0)|=\sup_{z\in \Dc_k'} |h_k(z_0)|$. We have $\tmop{area}_{\mathbb H^2}
(\mathcal R_k) \gg r^2/k^2$ uniformly w.r.t. $z_0$. Hence, repeating the argument used to deduce \eqref{eq:expsupbd1} shows that
\[
\mathbb E\left[\sup_{z\in  \Dc_k'} |h_k(z)|\right]  \ll  \frac{k^{2/p}}{r^{2/p}} \mathbb E \left[\left\Vert
h_k \right\Vert_{p,\Dc_k'}\right]+k^{7/4}r.
\]

By Theorem \ref{thm:expectationlpnorm}(ii.) along with Proposition \ref{prop:nearthecusp}(i.), and recalling $N=\frac{k}{12}+O(1)$, we have for any $p \ge 2$ with $p=o(k)$
\[
\mathbb{E}\left[\left\Vert h_{k}\right\Vert _{p,\Dc_k'}\right] \ll \sqrt{p} \left\Vert \sqrt {r_{k,\tmop{diag}}}\right\Vert_{p,\Dc_k'}+ \sqrt{p}k^{\frac14-\frac{3}{2p}} \ll k^{1/4} \sqrt{p}.
\]
Hence, we conclude that for any $r>0,p\ge 2$ that 
\[
\mathbb E\left[\sup_{z\in  \Dc_k'} |h_k(z)|\right]  \ll  \frac{k^{2/p}}{r^{2/p}} k^{1/4}\sqrt{p}+k^{7/4}r.
\]
Taking $r=k^{-7/4}$ and $p=\log k$ then recalling \eqref{eq:cuspbd} completes the proof.
\end{proof}


\begin{thebibliography}{99}

% \bibitem{abbes-ullmo-1995}
% Abbes, A. and Ullmo, E.
% Comparaison des m\'{e}triques d'Arakelov et de Poincar\'{e} sur $X_0(N)$.
% \emph{Duke Math. J.} 80 (1995), no. 2, 295--307.

\bibitem{Stats with Math}
Abell, M.L., Braselton, J.P. and Rafter, J.A.
Statistics with mathematica.
Academic Press (1999).

\bibitem{AdTa}
Adler, R.J. and Taylor, J.E.
Random fields and geometry.
Springer Science \& Business Media (2009).

\bibitem{abms} Aryasomayajula, A., Biswas, I., Morye, A. S., and Sengupta, T.  On the Kähler metrics over $\tmop{Sym}^d(X)$. \emph{J. Geom. Phys.}, 110, 187-194. (2016)

\bibitem{am} Aryasomayajula, A., and Mukherjee, A.. Estimates of Kähler metrics on noncompact finite volume hyperbolic Riemann surfaces, and their symmetric products. \emph{Ann. Global Anal. Geom.} 66 (3) (2024): 11.


% \bibitem{aryasomayajula-priyanka-2018} Aryasomayajula, A., and Priyanka M. Off-diagonal estimates of the Bergman kernel on hyperbolic Riemann surfaces of finite volume. \textit{Proc. Amer. Math. Soc.} 146.9 (2018): 4009-4020.

\bibitem{ma-marinescu-2021}
Auvray, H., Ma, X. and Marinescu, G.
Bergman kernels on punctured Riemann surfaces.
\emph{Math. Ann.} 379 (2021), no. 3-4, 951--1002.

% \bibitem{berry-1977}
% Berry, M. V.
% Regular and irregular semiclassical wavefunctions.
% \emph{J. Phys. A} 10 (1977), no. 12, 2083--2091.


\bibitem{BlShZe Inv}
Bleher, P., Shiffman, B. and Zelditch, S.
Universality and scaling of correlations between zeros on complex manifolds.
\emph{Invent. Math.} 142 (2000), no. 2, 351--395.

\bibitem{BurqLebeau}
Burq, N. and Lebeau, G.
Injections de Sobolev probabilistes et applications.
\emph{Ann. Sci. \'{E}c. Norm. Sup\'{e}r. }(4) 46
(2013), no. 6, 917--962.

\bibitem{cogdell-luo-2011}
Cogdell, J. W. and Luo, W.
The Bergman kernel and mass equidistribution on the Siegel modular variety.
\emph{Forum Math.} 23 (2011), no. 1, 141–159.

\bibitem{DrLiMar sup}
Drewitz, A., Liu, B. and Marinescu, G.
Large deviations for zeros of holomorphic sections on punctured Riemann surfaces.
\emph{Michigan Math. J.} 75 (2025), no. 2, 381--421.


\bibitem{DrLiMar Gauss}
Drewitz, A., Liu, B. and Marinescu, G.
Gaussian holomorphic sections on noncompact complex manifolds.
\emph{arXiv preprint} arXiv:2302.08426 (2023).

\bibitem{AdFeYa}
Feng, R., Yao, D. and Adler, R.J.
Critical radii and suprema of random waves over Riemannian manifolds.
\emph{arXiv preprint} arXiv:2501.10798 (2025).

\bibitem{friedman-jorgenson-kramer-2016}
Friedman, J.S., Jorgenson, J. and Kramer, J.
Uniform sup norm bounds on average for cusp forms of higher weights.
In: W. Ballmann et al. (eds.), ``Arbeitstagung Bonn 2013. In memory of
Friedrich Hirzebruch''. Progress in Mathematics, vol. 319, 127--154. Birkhauser, Basel, 2016.


\bibitem{gradshteyn-ryzhik-2014}
Gradshteyn, I. S.  and Ryzhik, I. M.
Table of integrals, series, and products.
Academic press, 2014.

% \bibitem{GAF book}
% Hough, J.B., Krishnapur, M. and Peres, Y.
% Zeros of Gaussian analytic functions and determinantal point processes (Vol. 51). A
% merican Mathematical Soc. (2009).

\bibitem{Iwaniec-2002}
Iwaniec, H.
Spectral Methods of Automorphic Forms. Vol. 53.
American Mathematical Soc., 2002.

\bibitem{jorgenson-kramer-2004}
Jorgenson, J. and Kramer, J.
Bounding the sup norm of automorphic forms.
\emph{Geom. Funct. Anal.} 14 (2004), no. 6, 1267--1277.

\bibitem{jorgenson-kramer-2011}
Jorgenson, J. and Kramer, J.
sup norm bounds for automorphic forms and Eisenstein series.
In: J. Cogdell et al. (eds.), ``Arithmetic Geometry and Automorphic Forms''. Advanced Lectures
in Mathematics, vol. 19, 407--444. Higher Education Press and International Press, Beijing-Boston, 2011.

\bibitem{Ledoux}
Ledoux, M.
The concentration of measure phenomenon.
\emph{Math. Surveys Monogr.} 89, American Mathematical Society,
Providence, RI, 2001. x+181 pp.



\bibitem{MaMar Adv}
Ma, X. and Marinescu, G.
Generalized Bergman kernels on symplectic manifolds.
\emph{Adv. Math.} 217 (2008), no. 4, 1756--1815.

\bibitem{ma-marinescu-2007} Ma, X. and Marinescu, G. Holomorphic Morse inequalities and Bergman kernels,
\emph{Progress in Mathematics}, 254. Birkhäuser Verlag, Basel, 2007.

\bibitem{michel-ullmo-1998}
Michel, Ph. and Ullmo, E.
Points de petite hauteur sur les courbes modulaires $X_0(N)$.
\emph{Invent. Math.} 131 (1998), no. 3, 645--674.

% \bibitem{milicevic-2010}
% Mili\'cevi\'c, D.
% Large values of eigenfunctions on arithmetic hyperbolic surfaces.
% \emph{Duke Math. J.} 155 (2010), no. 2, 365--401.

\bibitem{Rudnick}
Rudnick, Z.
On the asymptotic distribution of zeros of modular forms.
\emph{Int. Math. Res. Not.} 2005, no. 34, 2059--2074.

\bibitem{tian-1990} Tian, G. \emph{On a set of polarized Kähler metrics on algebraic manifolds, J. Differential Geom.}. 32 (1990), 99–130.

% \bibitem{sarnak-1993}
% Sarnak, P.
% Arithmetic quantum chaos.
% The R. A. Blyth Lectures, University of Toronto, Canada, 1993.

\bibitem{sarnak-2004}
Sarnak, P.
Letter to C. Morawetz, 2004.
Available at \url{https://publications.ias.edu/sites/default/files/Sarnak_Letter_to_Morawetz.pdf}{}.

\bibitem{ShZe GAFA}
Shiffman, B. and Zelditch, S.
Number variance of random zeros on complex manifolds.
\emph{Geom. Funct. Anal.} 18 (2008), no. 4, 1422--1475.

\bibitem{Steiner}
Steiner, R.S.
Uniform bounds on sup norms of holomorphic forms of real weight.
\emph{Int. J. Number Theory} 12 (2016), no. 5, 1163--1185.

\bibitem{xia-2007}
Xia, H.
On $L^{\infty}$ norms of holomorphic cusp forms.
\emph{J. Number Theory} 124 (2007), no. 2, 325--327.

\bibitem{zagier-1976}
Zagier, D.
The Eichler--Selberg Trace Formula on $SL_2 (\mathbb Z)$.
An appendix to the book ``Introduction to modular forms'' by S. Lang, Springer-Verlag, Berlin-New York, pages 44--55, 1976.

\bibitem{zagier-1977}
Zagier, D.
Modular forms whose Fourier coefficients involve zeta-functions of quadratic fields.
In Modular functions of one variable, VI (Proc. Second Internat. Conf., Univ. Bonn, Bonn, 1976), pages 105--169.
Lecture Notes in Math., Vol. 627, 1977.

 \bibitem{zelditch-1998} Zelditch, S.. Szegő kernels and a theorem of Tian, \emph{Internat. Math. Res. Notices}
1998, no. 6, 317–331.


\end{thebibliography}
\end{document}